\documentclass[12pt]{amsart}

\pdfoutput=1

\usepackage{latexsym}
\usepackage[centertags]{amsmath}
\usepackage{amsfonts}
\usepackage{amssymb}
\usepackage{amsthm}
\usepackage{newlfont}
\usepackage{graphics}
\usepackage{color}

\usepackage[font=small,labelfont=bf,margin=5mm]{caption}

\RequirePackage{xparse, graphicx, caption}

\usepackage[usenames,dvipsnames]{xcolor}
\usepackage{graphicx}

\usepackage{wrapfig}

\newtheorem{theo}{Theorem}
\newtheorem{defn}[theo]{Definition}
\newtheorem{exam}[theo]{Example}
\newtheorem{lem} [theo]{Lemma}
\newtheorem{cor}[theo]{Corollary}
\newtheorem{prop}[theo]{Proposition}
\newtheorem{rem}[theo]{Remark}

\makeatletter \@addtoreset{equation}{section}
\@addtoreset{theo}{}\makeatother

\def\None{\N\setminus \{1\}}

\def\N{\mathbb{N}}

\newcommand{\rank}{\operatorname{rank}}

\def\N{\mathbb{N}}

\title[An order matching for $L(3,n)$]{Constructing explicit Sperner chain decompositions for $L(3,n)$ and $L(4,n)$ via Greedy Algorithms and chain tableaux}

\author{Guoce Xin}

\address{School of Mathematical Sciences, Capital Normal University,
Beijing 100048, PR China}

\email{guoce\_xin@163.com}

\author{Yueming Zhong*}

\thanks{*Corresponding author}

\address{School of Mathematics and Statistics, Hainan University, Haikou 570228,
PR China}

\email{zhongyueming107@gmail.com}

\date{\today} 

\begin{document}
\begin{abstract}
Let $L(m,n)$ denote Young's lattice, consisting of all partitions whose Young diagrams are contained within an $m\times n$ rectangle. It is a classical result that the partially ordered set $L(m,n)$ is rank-symmetric, rank-unimodal, and Sperner; however, finding a direct combinatorial proof via an explicit order matching remains a prominent open problem in the field. In this paper, we address this challenge by constructing explicit order matchings for $L(3,n)$ and extending our methods to comprehensively cover $L(4,n)$. To achieve this, we introduce a novel ``chain tableau" representation, which serves as a powerful tool for identifying and characterizing complex combinatorial patterns. Notably, we demonstrate that the same order matchings can be independently derived using both a greedy algorithm and a recursive kneading process. This work not only resolves the explicit matching problem for $m=3$ and $m=4$ but also establishes robust structural tools that may offer valuable insights into the general $L(m,n)$ case.
\end{abstract}

\maketitle

\vspace{-5mm}
\maketitle

\noindent
\begin{small}
 \emph{Mathematic subject classification}: Primary 05A19; Secondary 05A17; 06A07.
\end{small}

\noindent
\begin{small}
\end{small}

\section{Introduction \label{sec:Introduction}}

For positive integers $m$ and $n$, let $L(m,n)$ denote the set of all partitions $\lambda=(\lambda_1,\lambda_2,...,\lambda_m)$ such that $n\geqslant  \lambda_1 \geqslant \lambda_2 \geqslant \cdots  \lambda_m \geqslant 0$.
That is, $L(m,n)$ consists of all partitions $\lambda$ whose Young diagram fits in the $m\times n$ rectangle. For $\lambda, \mu \in L(m,n)$, define
$\lambda \leqslant \mu$ if $\lambda_i\leqslant \mu_i $ for all $1\leqslant i\leqslant m$. In other words, the Young diagram of $\lambda$ is contained in that of $\mu$. Then $L(m,n)$ is a partially ordered set. Indeed, it is a distributive lattice of cardinality $\tbinom{m+n}{m}$.

It is also easy to see that $L(m,n)$ is a ranked poset. That is, it satisfies the Jordan-Dedekind chain condition: all maximal chains between two comparable elements have the same length. The length of a maximal chain between $(0,0,...,0)$ and $\lambda=(\lambda_1,\lambda_2,...,\lambda_m)$ is $|\lambda|=\lambda_1 + \lambda_2 + ... + \lambda_m$. This is the rank $r(\lambda)$ of
$\lambda$. The dual of $\lambda$ in $L(m,n)$, denoted $\lambda^*=(n-\lambda_m,\dots, n-\lambda_2,n-\lambda_1)$, is an involution and plays an important role.
The rank generating function of $L(m,n)$ is the well-known
 Gaussian polynomial:
  \begin{equation}
 \sum_{i\geqslant0} p_i(m,n)q^i =  \tbinom{\textbf{m+n}}{\textbf{m}} = \frac{(1-q)(1-q^2)\cdots(1-q^n)}{(1-q)(1-q^2)\cdots(1-q^k)(1-q)(1-q^2)\cdots(1-q^{n-k})},
 \end{equation}
 where $p_i(m,n)$ is the number of elements in $L_i(m,n)$, which is the set of all elements of rank $i$ in $L(m,n)$. See \cite[Chap. 6]{Stanley-Book}.

A classical result about $L(m,n)$ is the following.
\begin{theo}\label{t-Lmn}
The posets $L(m,n)$ are rank symmetric, rank unimodal, and Sperner.
\end{theo}
The rank symmetry follows easily from the dual operation. The rank unimodality is due to Sylvester \cite{Sylvester}. Later
O'hara  gave a combinatorial proof that astounded the combinatorial community. See \cite{Ohara} and \cite{Doron}. Stanley gave a nice survey on unimodality and log-concavity in
\cite{Stanley-LogConcave}. The spernicity
was first proved by Stanley \cite{Stanley-Spernicity}.
See \cite[Corollary 6.10]{Stanley-Book} for a proof and reference therein.

It was shown that there exists an \emph{order matching}
$\varphi_i$ from $L_i(m,n)$ to $L_{i+1}(m,n)$ for each $i<mn/2$. That is, $\varphi_i$ is an injection and satisfies $\varphi_i(\lambda)>\lambda$
for all $\lambda\in L_i(m,n)$. All known proofs of this result used algebra technique.

It is still an outstanding combinatorial problem to find an explicit order matching $\varphi_i$ from $L_i(m,n)$ to $L_{i+1}(m,n)$. The problem is
equivalent to finding a Sperner chain decomposition, which is usually simpler to describe than the order matching. A
a chain decomposition $\{ C_j \}_{1\le j\le N}$ of $L(m,n)$ is called \emph{Sperner} if each chain contains an element of middle rank $\lfloor mn/2 \rfloor$.
Such a chain decomposition clearly implies the Sperner property.
See Section \ref{sec:RecursiveSperner} for detailed discussion.

Symmetric chain decompositions requires each chain is rank symmetric. Its existence for $L(m,n)$ was not confirmed for $m\ge 5$.
Symmetric chain decompositions was only constructed for $m=3$ in \cite{B.Lindstrom}, for $m=4$ in \cite{Nathan}, and for $m=3,4$
in \cite{Wen}.
Consequently, order matchings for $m=3,4$ are known \cite{John}, but explicit order matching even for $m=3$ does not seem to appear in the literature.

Our ultimate goal is to construct explicit Sperner chain decompositions for $L(m,n)$. By using greedy algorithm, computer experiment suggests
Sperner chain decompositions for $L(3,n)$ and $L(4,n)$, but not for larger $m$. The pattern for $L(3,n)$ is particularly nice, and we obtain
an explicit order matching as stated in Theorem \ref{theo-varphi}. Analogous result for $L(4,n)$ seems too complicated to be presented.

We will give two equivalent descriptions of Theorem \ref{theo-varphi} using Sperner chain decompositions:  i) the idea is to
represent a chain by a standard Young tableau of a skew shape, which we call a chain tableau. We will see the pattern easily from the chain tableaux
in Section \ref{sec:L3nVarphi}. ii) the idea is to use the recursive structure $L(m,n)=L(m,n-1)\biguplus (n\oplus L(m-1,n))$ and a kneading method.
Both ideas extend for $L(4,n)$, but need fresh ideas for larger $m$.

The paper is organized as follows.
Section \ref{sec:Introduction} is this introduction.
In Section \ref{sec:varphi}, we give an explicit order matching $\varphi$ for $L(3,n)$ in Theorem \ref{theo-varphi}, and show that it agrees with the greedy algorithm.
In Section \ref{sec:L3nVarphi}, we give a tableaux version of $\varphi$ in Theorem \ref{theo:varphi-chain}.  The proof relies on the properties of $\varphi$.
Section \ref{sec:L3nL4n} presents a self contained proof of Theorem \ref{theo:varphi-chain}. The idea extends for $L(4,n)$. The corresponding result is stated in Theorem \ref{theo:L4n-chain-dec}.
We only outlined the proof.
Section \ref{sec:RecursiveSperner} constructs the Sperner chain decompositions of $L(m,n)$ for $m=3,4$ by a recursive method.

\section{The order matching $\varphi$ for $L(3,n)$ and the greedy algorithm \label{sec:varphi}}
We first state and prove our order matching for $L(3,n)$, and then talk about its discovery by the greedy algorithm.

\subsection{The order matching $\varphi$}
The order matching relies on the starting partitions and end partitions defined by:
\begin{align}
  S_{3,n} &=\{ (4k+\ell,2k,0): k \in \N, \ell \in \N\setminus\{1\}, 4k+\ell \leqslant n, 6k+\ell \leqslant 3n/2 \} \\
  E_{3,n} &= S_{3,n}^* = \{ \lambda^*: \lambda \in S_{3,n}\}. \label{e-E3n}
\end{align}
Then partitions in $S_{3,n}$ have ranks no more than $3n/2$, and partitions in $E_{3,n}$ have ranks no less than $3n/2$.
We call $S_{3,n}$ the starting set and $E_{3,n}$ the end set.

The following lemma will be frequently used without mentioning.

\begin{lem}\label{lem-E3n-check}
A partition $\lambda=(\lambda_1,\lambda_2,\lambda_3)$ belongs to $E_{3,\lambda_1}$ if and only if:
\begin{enumerate}
  \item $\lambda_1-\lambda_2$ is even.
  \item $2\lambda_2-\lambda_1-\lambda_3\in \None$.
\end{enumerate}
\end{lem}

\begin{theo}\label{theo-varphi}
The following map $\varphi$ is a bijection from $L(3,n)\setminus E_{3,n}$ to $ L(3,n)\setminus S_{3,n}$.
$$ \varphi(\lambda)=\left\{
\begin{array}{rcl}
(\lambda_1 +1,\lambda_2,\lambda_3),      &      &  \text{if } \lambda \in E_{3,\lambda_1},    \text{  otherwise };\\
(\lambda_1,\lambda_2+1,\lambda_3),       &      & \text{if } \lambda_2+\lambda_3 \equiv0 \pmod2 \text{ and } (\lambda_1 -1,\lambda_2+1,\lambda_3) \notin E_{3,\lambda_1 -1};\\
(\lambda_1,\lambda_2+1,\lambda_3),       &      &\text{if }\lambda_2+\lambda_3 \not\equiv 0 \pmod2 \text{ and } (\lambda_1 -1,\lambda_2,\lambda_3+1) \in E_{3,\lambda_1 -1}; \\
(\lambda_1,\lambda_2,\lambda_3+1),       &      & \text{if } \lambda_2+\lambda_3 \not\equiv 0 \pmod2 \text{ and } (\lambda_1 -1,\lambda_2,\lambda_3+1) \notin E_{3,\lambda_1 -1}.
\end{array}
\right. $$
Therefore $\varphi$ defines order matchings
$$L_0(3,n) \to L_1(3,n) \to \cdots L_{j} (3,n) \leftarrow L_{j+1} (3,n) \leftarrow \cdots \leftarrow L_{3n}(3,n), $$
where $j=\lfloor 3n/2\rfloor$. Then $L(3,n)$ is rank unimodal and Sperner.
\end{theo}

\begin{rem}
For a full classification of the case $\lambda\not\in E_{3,\lambda_1}$, we need to consider the case $\lambda_2+\lambda_3 \equiv 0 \pmod2 \text{ and } (\lambda_1 -1,\lambda_2+1,\lambda_3) \in E_{3,\lambda_1 -1}$.
But there is indeed no partition belonging to this case. That is, there is no partition $\lambda$ satisfying
(i) $\lambda \notin E_{3,\lambda_1}$; (ii) $\lambda_2+\lambda_3\equiv 0 \pmod 2$; and (iii) $\mu=(\lambda_1-1,\lambda_2+1,\lambda_3) \in E_{3,\lambda_1-1}$.
The reason is as follows.

By (iii), we have $\mu_1-\mu_2=\lambda_1-\lambda_2-2$ is even, and $2\mu_2-\mu_1-\mu_3=2\lambda_2-\lambda_1-\lambda_3+3\in \None$. By (i) and $\lambda_1-\lambda_2$ is even, we have
$2\lambda_2-\lambda_1-\lambda_3 \notin \None$. Then $2\lambda_2-\lambda_1-\lambda_3 \in \{-3,-1,1\}$. But this contradicts the fact that  $\lambda_1+\lambda_3=(\lambda_1-\lambda_2)+(\lambda_2+\lambda_3)$ is even.
\end{rem}

The proof follows from the following two lemmas, since $E_{3,n}$ and $S_{3,n}$ have the same cardinality.

\begin{lem}
  The map $\varphi$   is well-defined from $L(3,n)\setminus E_{3,n}$ to $ L(3,n)\setminus S_{3,n}$.
\end{lem}

\begin{proof}
Firstly, we show $\varphi(\lambda)$ is a valid partition. We discuss in 3 cases as follow.

\begin{enumerate}
  \item [case 1.] If $\varphi(\lambda)=(\lambda_1+1,\lambda_2,\lambda_3)$ then we need to show $\lambda_1 < n$. But if $\lambda_1=n$
  then $\lambda\in E_{3,\lambda_1}$ conflicts $\lambda \in L(3,n)\setminus E_{3,n}$.

  \item [case 2.] If $\varphi(\lambda)=(\lambda_1,\lambda_2+1,\lambda_3)$, then we need to show that $\lambda_1> \lambda_2$. Suppose
  to the contrary that $\lambda_1=\lambda_2$ so that $\lambda_1-\lambda_2=0$ is even. Then by Lemma \ref{lem-E3n-check} we have
  $$\lambda \notin E_{3,\lambda_1}\Rightarrow \lambda_1-\lambda_3 \not\in \None \Rightarrow \lambda_1-\lambda_3=1.$$
  Now $\lambda_2+\lambda_3=2\lambda_1-1$ is odd, and $(\lambda_1-1,\lambda_2,\lambda_3+1)
  =(\lambda_1-1,\lambda_1,\lambda_1)$ is not a partition (so not in $E_{3,\lambda_1-1}$).
      This contradicts the definition of $\varphi$.

  \item [case 3.] If $\varphi(\lambda)=(\lambda_1,\lambda_2,\lambda_3+1)$, then we need to show that $\lambda_2 > \lambda_3$. Suppose to the contrary that $\lambda_2=\lambda_3$.
   Since $\lambda_2+\lambda_3=2\lambda_2$ is even, then by definition of $\varphi$, we have $\rho=(\lambda_1-1,\lambda_2+1,\lambda_2)\in E_{3,\lambda_1-1}$, which implies that $\rho_1-\rho_2=\lambda_1-\lambda_2-2$ is even and
  $2\rho_2-\rho_1-\rho_3= 3-(\lambda_1-\lambda_2)\in \None$. Clearly, no such $\lambda_1-\lambda_2$ exists.
\end{enumerate}

Secondly, we show $\varphi(\lambda)\notin S_{3,n}$ by contradiction. If $\varphi(\lambda)\in S_{3,n}$ then there exists $k,\ell(\neq 1) \in \mathbb{N}^+$ such that $\varphi(\lambda)=(4k+\ell,2k,0)$. We discuss in 3 cases as follow.

\begin{enumerate}
  \item [Case 1.] If $\varphi(\lambda)=(\lambda_1+1,\lambda_2,\lambda_3)$ then $\lambda=(4k+\ell-1,2k,0)$. Now $\lambda\in E_{3,\lambda_1}$
  implies that $\lambda_1-\lambda_2=2k+\ell-1$ is even and $2\lambda_2-\lambda_1-\lambda_3=1 -\ell \in \None \Rightarrow \ell=1$. This is clearly impossible.
  \item [Case 2.] If $\varphi(\lambda)=(\lambda_1,\lambda_2+1,\lambda_3)$ then $\lambda=(4k+\ell,2k-1,0)$. By $\lambda_2+\lambda_3=2k-1$, we get $\rho=(\lambda_1-1,\lambda_2,\lambda_3+1)\in E_{3,\lambda_1-1}$, which implies that $2\rho_2-\rho_1-\rho_3=-\ell-2\in \None \Rightarrow \ell \le -2$. This is clearly impossible.
  \item [Case 3.] If $\varphi(\lambda)=(\lambda_1,\lambda_2,\lambda_3+1)$ then $\lambda=(4k-1-\ell,2k,-1)$. We get $\lambda$ is not a valid partition.
\end{enumerate}
 \end{proof}

\begin{lem}
 The the map $\varphi$ is one-to-one.
\end{lem}

\begin{proof}
We prove by contradiction. Suppose there is a partition
 $\lambda=(\lambda_1,\lambda_2,\lambda_3) \in L(3,n)  $ such that
 $\varphi(\mu)=\varphi(\nu)=\lambda$. Then $\mu$ and $\nu$ are both obtained from $\lambda$ by
 removing one of its corners. This is divided in to 3 cases. \par
\begin{enumerate}
\item [\textbf{Case 1:}] Let $\mu=(\lambda_1 -1,\lambda_2,\lambda_3)$, $\nu=(\lambda_1,\lambda_2 -1,\lambda_3)$ with $\varphi(\mu)=\varphi(\nu)=\lambda$.

On one hand, $\varphi(\mu)=\lambda$ implies that $\mu\in E_{3,\lambda_1-1}$.
Thus $\mu_1-\mu_2=\lambda_1 -\lambda_2-1$ is even;

On the other hand, $\varphi(\nu)=\lambda$ implies that: i) If $\nu_2 + \nu_3=(\lambda_2 -1)+\lambda_3$ is even,
then $\rho=(\nu_1-1,\nu_2+1,\nu_3)=(\lambda_1-1,\lambda_2,\lambda_3)=\mu \notin E_{3,\lambda_1-1}$, which conflicts with $\mu\in E_{3,\lambda_1-1}$.
ii) If $\nu_2 + \nu_3=(\lambda_2 -1)+\lambda_3$ is odd, then $\omega=(\nu_1-1,\nu_2,\nu_3+1)=(\lambda_1-1,\lambda_2-1,\lambda_3+1)\in E_{3,\nu_1-1}$.
We get $\omega_1-\omega_2=\lambda_1-\lambda_2$ is even,
which conflicts with that $\mu_1-\mu_2=\lambda_1-\lambda_2-1$ is even.

\item [\textbf{Case 2:}] Let $\mu=(\lambda_1 -1,\lambda_2,\lambda_3), \nu=(\lambda_1,\lambda_2,\lambda_3-1)$ with $\varphi(\mu)=\varphi(\nu)=\lambda$.
On one hand, $\varphi(\mu)=\lambda$ implies that $\mu \in E_{3,\lambda_1-1}$; On the other hand,
$\varphi(\nu)=\lambda$ implies that $\rho=(\nu_1-1,\nu_2,\nu_3+1)=(\lambda_1-1,\lambda_2,\lambda_3)=\mu \notin E_{3,\lambda_1-1}$. This is a contradiction.

\item [\textbf{Case 3:}] Let $\mu=(\lambda_1,\lambda_2-1,\lambda_3), \nu=(\lambda_1,\lambda_2,\lambda_3-1)$ with $\varphi(\mu)=\varphi(\nu)=\lambda$.
Now $\varphi(\nu)=\lambda$ implies that $\nu_2+\nu_3(=\mu_2+\mu_3) \not\equiv 0\pmod2$. That is, $\lambda_2+\lambda_3$ is even.
Note also that $\nu \not\in E_{3,\nu_1}$ implies that either $\nu_1-\nu_2$ is odd or $2\nu_2-\nu_1-\nu_3\not\in \None$.

But by $\varphi(\mu)=\lambda$, we have $\omega=(\mu_1-1,\mu_2,\mu_3+1)\in E_{3,\mu_1-1}$. This implies that $\omega_1-\omega_2=\lambda_1-\lambda_2=\nu_1-\nu_2$ is even, and $2\omega_2-\omega_1-\omega_3=2\lambda_2-\lambda_1-\lambda_3-2 \in \None$. Hence $2\nu_2-\nu_1-\nu_3=2\lambda_2-\lambda_1-\lambda_3+1 \in \None$.
\end{enumerate}
 \end{proof}

\subsection{The involutions $^*\varphi$ and $\varphi^*$}
The two involutions we discovered indeed state that the following is an identity map:
$$ L(3,n)\setminus E_{3,n}  \mathop{\longrightarrow}\limits^\varphi L(3,n)\setminus S_{3,n}  \mathop{\longrightarrow}\limits^*
L(3,n)\setminus E_{3,n}  \mathop{\longrightarrow}\limits^\varphi L(3,n)\setminus S_{3,n}  \mathop{\longrightarrow}\limits^* L(3,n)\setminus E_{3,n}.$$
More precisely, we have the following.
\begin{theo}
The map $^*\varphi$ defined by $^*\varphi(\lambda)= [\varphi(\lambda)]^*$ is an involution from $L(3,n)\setminus E_{3,n}$ to itself. That is,
for any $\lambda \in L(3,n)\setminus E_{3,n}    $, we have $[\varphi([\varphi(\lambda)]^*)]^*=\lambda$, or equivalently, $\varphi([\varphi(\lambda)]^*)=\lambda^*$.
\end{theo}
\begin{proof}
Let $\mu=[\varphi(\lambda)]^*$. We show that $\varphi(\mu)=\lambda^*$ by the following three cases.

Case 1: $\varphi(\lambda)=(\lambda_1+1,\lambda_2,\lambda_3)$ so that $\mu=(n-\lambda_3,n-\lambda_2,n-\lambda_1-1)$. This only happen when $\lambda\in E_{3,\lambda_1}$.
We need to show that $\mu\not\in E_{3,\mu_1}$, $\mu_2+\mu_3$ is odd, and $\rho=(\mu_1-1,\mu_2,\mu_3+1)\not\in E_{3,\mu_1-1}$.
By $\lambda\in E_{3,\lambda_1}$, $\lambda_1-\lambda_2$ is even, and $2\lambda_2-\lambda_1-\lambda_3\in \None$. To show that
$\mu\not\in E_{3,\mu_1}$, we observe that $2\mu_2-\mu_1-\mu_3=\lambda_1+\lambda_3-2\lambda_2+1 \left(\in \{1,-1,-2,\dots\} \right)$ cannot belong to
$\None$; $\mu_2+\mu_3=2n-\lambda_1-\lambda_2-1$ is odd; $2\rho_2-\rho_1-\rho_3=\lambda_1+\lambda_3-2\lambda_2+1$ cannot belong to $\None$ again.

\medskip
Case 2: $\varphi(\lambda)=(\lambda_1,\lambda_2+1,\lambda_3)$ so that $\mu=(n-\lambda_3,n-\lambda_2-1,n-\lambda_1)$.
We need to show that $\varphi(\mu)=(\mu_1,\mu_2+1,\mu_3)=\lambda^*$. By definition of $\varphi$, we have to consider the following two subcases.

\begin{enumerate}

\item [a)] When $\lambda_2+\lambda_3$ is even, $\lambda\notin E_{3,\lambda_1}$, and $(\lambda_1-1,\lambda_2+1,\lambda_3)\notin E_{3,\lambda_1-1}$:
Firstly $\mu\not\in E_{3,\mu_1}$ since $\mu_1-\mu_2=\lambda_2-\lambda_3+1$ is odd. Next we have to consider the following two cases.

\begin{enumerate}
  \item [(i)] If $\mu_2+\mu_3$ is even, then we need to show that $\rho=(\mu_1-1,\mu_2+1,\mu_3)\notin E_{3,\mu_1-1}$. This is obvious since $\rho_1-\rho_2=\lambda_2-\lambda_3-1$ is  odd.
  \item [(ii)] If $\mu_2+\mu_3=2n-\lambda_1-\lambda_2-1$ is odd, then we need to show that $\rho=(\mu_1-1,\mu_2,\mu_3+1)\in E_{3,\mu_1-1}$.
Firstly, $\rho_1-\rho_2=\mu_1-\mu_2-1=\lambda_2-\lambda_3$ is even. Next we show that $2\rho_2-\rho_1-\rho_3=\lambda_1+\lambda_3-2\lambda_2-2\in \None$.
By $\lambda \notin E_{3,\mu_1}$ and $\lambda_1-\lambda_2\equiv 0 \pmod2$, we have $2\lambda_2-\lambda_1-\lambda_3\notin \None$. This implies that
$\lambda_1+\lambda_3-2\lambda_2-2 \in \{-3,-1,0,1,2,\dots\}$. The proof is then completed by the fact that $\lambda_1+\lambda_3=(\lambda_2+\lambda_3)+(\lambda_1-\lambda_2)$ is even.
\end{enumerate}

\item [b)] When $\lambda_2+\lambda_3$ is odd, $\lambda\notin E_{3,\lambda_1}$, and $\omega=(\lambda_1-1,\lambda_2,\lambda_3+1)\in E_{3,\lambda_1-1}$:
We get $\omega_1-\omega_2=\lambda_1-\lambda_2-1$ is even. Then $\mu_2+\mu_3=2n-\lambda_1-\lambda_2-1$ is even. So we need to that $\mu \notin E_{3,\mu_1}$ and $\rho=(\mu_1-1,\mu_2+1,\mu_3)\notin E_{3,\mu_1-1}$,
we have $2\omega_2-\omega_1-\omega_3=2\lambda_2-\lambda_1-\lambda_3 \in \None$. Thus $2\rho_2-\rho_1-\rho_3=\lambda_1+\lambda_3-2\lambda_2+1 \notin \None$ and $2\mu_2-\mu_1-\mu_3=\lambda_1+\lambda_3-2\lambda_2-2 \notin \None$.
Hence $\mu \notin E_{3,\mu_1}$ and $\rho\notin E_{3,\mu_1-1}$.

\end{enumerate}

\medskip
Case 3: $\varphi(\lambda)=(\lambda_1,\lambda_2,\lambda_3+1)$ so that $\mu=(n-\lambda_3-1,n-\lambda_2,n-\lambda_1)$. This only happens when $\lambda \not\in E_{3,\lambda_1} $, $\lambda_2+\lambda_3$ is odd, and $\rho=(\lambda_1 -1,\lambda_2,\lambda_3+1) \notin E_{3,\lambda_1 -1}$. We need to show that $\mu \in E_{3,\mu_1}$. Firstly $\mu_1-\mu_2=\lambda_2-\lambda_3-1$ is even. Secondly by $\rho=(\lambda_1 -1,\lambda_2,\lambda_3+1) \notin E_{3,\lambda_1 -1}$ and $\lambda \notin E_{3,\lambda_1}$, we get $2\rho_2-\rho_1-\rho_3=2\lambda_2-\lambda_1-\lambda_3\notin \None$. Thus $2\mu_2-\mu_1-\mu_3=\lambda_1+\lambda_3-2\lambda_2+1 \in \None$.
\end{proof}

Similarly, we have the following.
\begin{cor}
The map $\varphi^*$ defined by $\varphi^*(\lambda)= \varphi(\lambda^*)$ is an involution from $L(3,n)\setminus S_{3,n}$. 
\end{cor}

Thus we have an alternative way to compute $\varphi^{-1}$ by $\varphi^{-1}(\mu)=\left[\varphi(\mu^*)  \right]^*$.

\subsection{The Greedy Algorithm}
Our discovery of $\varphi$ results from the greedy algorithm, which approximates a global
optimal solution by a local optimal solution.

Below we describe explicitly how to use the greedy algorithm to find a possible order matching from $L_i(m,n)$ to $L_{i+1}(m,n)$.

\noindent
\textbf{Algorithm GA}

\noindent
\textbf{Input}: $L_i(m,n)$ and $L_{i+1}(m,n)$ ordered increasingly. Here we choose the lexicographic order, still denoted ``$\leq$": For $\lambda=(\lambda_1,\lambda_2,\cdots,\lambda_m)\ne \mu=(\mu_1,\mu_2,\cdots,\mu_m)\in L_i(m,n)$,
find the smallest $s$ such that $\lambda_s\ne \mu_s$, then $\lambda>\mu$ if $\lambda_s>\mu_s$ and $\lambda<\mu$ if $\lambda_s<\mu_s$.
For instance, in $L_5(3,5)$ we have $221<311<32<41<5$.

\noindent
\textbf{Output}: Associate each $\lambda^j\in L_i(m,n)$ at most one $\mu^j\in L_{i+1}(m,n)$ denoted $GA(\lambda^j)$ that covers $\lambda^j$ (in the Young's lattice).
Note that we allow $GA(\lambda^j)$ does not exist. If every $\lambda^j$ has a $GA$ image, then
$GA$ is an injection; If every $\mu^j$ has a pre-image, then $GA$ is a surjection; Otherwise, we only obtain a partial matching.

Assume partitions in $L_i(m,n)$ are listed as
$\lambda^1< \cdots< \lambda^N$.
We successively construct $\mu^j$ for $j$ from 1 to $N$ such that $\mu^j$ cover $\lambda^j$ as follows.

For $\lambda^1$, match it with the smallest partition $\mu^1\in L_{i+1}(m,n)$ that covers $\lambda^1$.

Suppose that we have constructed $\mu^1,\dots, \mu^{j-1}$.
Then we greedily match $\lambda^j$ with the smallest partition $\mu^j$ that covers $\lambda^j$ and is not in $\{\mu^1,\dots, \mu^{j-1}\}$.
When no such $\mu^j$ exists, we say $GA(\lambda^j)$ does not exist.

\medskip
The greedy algorithm is easy to perform by computer. We find desired order matching for $m\le 4$, but fail for $m\ge 5$.

\begin{prop}
The map $\varphi$ agrees with the greedy algorithm:
\begin{enumerate}
\item $GA(\lambda)$ is well defined if and only if $\lambda \not\in E_{3,n}$.

\item If $\lambda \not\in E_{3,n}$ then $\varphi(\lambda)=GA(\lambda)$.
\end{enumerate}
\end{prop}
\begin{proof}
We prove by induction on $j$ that $\varphi(\lambda^j)=GA(\lambda^j)$.

The base case is routine:
Suppose $i=3q+r$ with $0\le r\le 2$. Then according to $r=0,1,2$, $\lambda^1$ equals $(q,q,q)$, $(q+1,q,q)$, $(q+1,q+1,q)$, with their $\varphi$ images
$(q+1,q,q)$, $(q+1,q+1,q)$ $ (q+1,q+1,q+1)$, respectively. These are exactly the choices of the Greedy Algorithm. %

Now assume by induction that $\varphi(\lambda^j)=GA(\lambda^j)$ for $j\le M-1$.
By Theorem \ref{theo-varphi} $\varphi(\lambda)$ is a valid partition, and
                   $\varphi(\lambda)\neq \varphi(\lambda^i)$ for $i<M$. We need to show that
this is exactly the choice of the Greedy Algorithm.

It is convenient to use the following notation according to the definition of $\varphi$:
\begin{align*}
  F_1=&\{\lambda: \lambda \in E_{3,\lambda_1}  \};\\
  F_{2}^e=&\{\lambda:\lambda_2+\lambda_3 \equiv 0 \pmod2 \text{ and }  (\lambda_1 -1,\lambda_2+1,\lambda_3)\notin E_{3,\lambda_1 -1} \text{ and } \lambda \notin E_{3,\lambda_1}\};\\
  F_2^o=& \{\lambda:\lambda_2+\lambda_3 \not\equiv 0 \pmod2 \text{ and }  (\lambda_1 -1,\lambda_2,\lambda_3+1) \in E_{3,\lambda_1 -1} \text{ and } \lambda \notin E_{3,\lambda_1}\}; \\
  F_3^o=&\{\lambda:\lambda_2+\lambda_3 \not\equiv 0 \pmod2 \text{ and } (\lambda_1 -1,\lambda_2,\lambda_3+1) \notin E_{3,\lambda_1 -1} \text{ and } \lambda \notin E_{3,\lambda_1}\}.\\
\end{align*}

To show that
$\varphi(\lambda)=GA(\lambda)$ for $\lambda=\lambda^M$, we consider the following four cases.

\begin{enumerate}
  \item[($F_3^o$)] If $\lambda \in F_3^o$, then the first choice of $GA(\lambda)$ is $(\lambda_1,\lambda_2,\lambda_3+1)=\varphi(\lambda)$.

  \item[($F_2^e$)] If $\lambda \in F_2^e$, then $GA(\lambda)=(\lambda_1,\lambda_2+1,\lambda_3)=\varphi(\lambda)$, because the first choice is occupied. That is,
we have $GA(\mu)=(\lambda_1,\lambda_2,\lambda_3+1)$, where $\mu<\lambda$ is given by
                $$ \mu=\left\{
                    \begin{array}{ll}
                              (\lambda_1-1,\lambda_2,\lambda_3+1), & \mbox{if  $\rho\in E_{3,\lambda_1}$;} \\
                              (\lambda_1,\lambda_2-1,\lambda_3+1), & \mbox{if  $\rho\not\in E_{3,\lambda_1}$,}
                   \end{array}
                    \right. \  \text{ where } \rho=(\lambda_1,\lambda_2-1,\lambda_3+1). $$

                   (I) If $\rho\in E_{3,\lambda_1}$, then $\rho_1-\rho_2=\lambda_1-\lambda_2+1$ is even,
                   and $2\rho_2-\rho_1-\rho_3=2\lambda_2-\lambda_1-\lambda_3-3 \in \None.$ To see that $\mu=(\lambda_1-1,\lambda_2,\lambda_3+1) \in E_{3,\mu_1}$ we check that: $\mu_1-\mu_2=\lambda_1-\lambda_2-1$ is even and $2\mu_2-\mu_1-\mu_3=2\lambda_2-\lambda_1-\lambda_3 \in \None$.

                   (II) If $\rho\not\in E_{3,\lambda_1}$, then we need to show that $\mu\in F_2^e$.
Firstly
$\mu_2+\mu_3=\lambda_2+\lambda_3$ is even by $\lambda\in F_2^e$; Secondly $\mu=\rho\notin E_{3,\lambda_1}$ is obvious;
Finally to show that $\omega=(\mu_1-1,\mu_2+1,\mu_3)\notin E_{3,\mu_1-1}$ we check that: i) If $\lambda_1-\lambda_2$ is even, then $\omega_1-\omega_2=\lambda_1-\lambda_2-1$ is odd so that $\omega\notin E_{3,\omega_1}$;
ii) If $\lambda_1-\lambda_2$ is odd, then $\mu_1-\mu_2=\lambda_1-\lambda_2+1$ is even, which implies, by $\mu\notin E_{3,\lambda_1}$ and the fact that $\lambda_1+\lambda_3=(\lambda_2+\lambda_3)+(\lambda_1-\lambda_2)$ is odd,
that  $2\mu_2-\mu_1-\mu_3=2\lambda_2-\lambda_1-\lambda_3-3\in \{1, -1,-2,-3,\dots\} \cap 2\mathbb{Z} = \{-2,-4,-6,\dots\} $.
Thus $2\omega_2-\omega_1-\omega_3=2\lambda_2-\lambda_1-\lambda_3\in \{1,-1,-3,\dots\}$. This shows that $\omega\not\in E_{3,\omega_1}$.

  \item[($F_2^o$)] If $\lambda \in F_2^o$, then $\mu=(\lambda_1-1,\lambda_2,\lambda_3+1) \in E_{3,\mu_1}$, so that $GA(\mu)=(\lambda_1,\lambda_2,\lambda_3+1).$ It follows that
 $GA(\lambda)=(\lambda_1,\lambda_2+1,\lambda_3)=\varphi(\lambda)$.

  \item[($F_1$)] If $\lambda \in F_1$, then $\lambda=(\lambda_1,\lambda_1-2k,\lambda_1-4k-\ell) \in E_{3,\lambda_1}$ for $\ell \in \None$. We have $GA(\lambda)=(\lambda_1+1,\lambda_2,\lambda_3)=\varphi(\lambda)$
                 (which is treated as not defined when $\lambda_1=n$), because the first and second choices are both occupied:
\begin{align*}
GA(\mu)&=(\lambda_1,\lambda_2+1,\lambda_3), \text{ where } \mu=(\lambda_1-1,\lambda_2+1,\lambda_3)\in F_1 \text{ and } \mu<\lambda;\\
GA(\nu)&=(\lambda_1,\lambda_2,\lambda_3+1), \text{ where } \nu=(\lambda_1,\lambda_2-1,\lambda_3+1)\in F_2^e\cup F_2^o \text{ and } \nu<\lambda.
\end{align*}

To show $\mu \in E_{3,\mu_1}$, we check that: by $\lambda\in E_{3,\lambda_1}$, we have $\mu_1-\mu_2=\lambda_1-\lambda_2-2$ is even  and $2\mu_2-\mu_1-\mu_3=2\lambda_2-\lambda_1-\lambda_3+3\in \None$.

To show $\nu\in F_2^e\cup F_2^o$, we consider the following two cases:

i) When $\lambda_2+\lambda_3 \equiv 0 \pmod2$, we need to show that $\nu \in F_2^e$. Firstly $\nu_2+\nu_3=\lambda_2+\lambda_3$ is even; Secondly $\lambda_1-\lambda_2$ is even by $\lambda\in E_{3,\lambda_1}$. Thus $\nu_1-\nu_2=\lambda_1-\lambda_2+1$ is odd, so that $\nu\notin E_{3,\lambda_1}$; Finally $\rho=(\nu_1-1,\nu_2+1,\nu_3)\notin E_{3,\lambda_1-1}$ follows from the fact that $\rho_1-\rho_2=\lambda_1-\lambda_2-1$ is odd.

ii) When $\lambda_2+\lambda_3 \not\equiv 0 \pmod2$, we need to show that $\nu \in F_2^o$.
Firstly $\nu_2+\nu_3=\lambda_2+\lambda_3$ is odd; Secondly by $\lambda\in E_{3,\lambda_1}$,
we get $\lambda_1-\lambda_2$ is even. Thus $\nu_1-\nu_2=\lambda_1-\lambda_2+1$ is odd, so that $\nu\notin E_{3,\lambda_1}$;
Finally for $\rho=(\nu_1-1,\nu_2,\nu_3+1)\in E_{3,\rho_1}$ we check that: a) $\rho_1-\rho_2= \lambda_1-\lambda_2$ is even; b)
since $2\lambda_2-\lambda_1-\lambda_3\in \None$ and $\lambda_1+\lambda_3=(\lambda_2+\lambda_3)+(\lambda_1-\lambda_2)$ is odd, $2\rho_2-\rho_1-\rho_3=2\lambda_2-\lambda_1-\lambda_3-3 \in \None$.
\end{enumerate}
\end{proof}

\section{The chain tableau of $L(3,n)$ with respect to $\varphi$ \label{sec:L3nVarphi}}
It is standard to represent a chain $C:\mu_1 \lessdot \mu_2 \cdots \lessdot \mu_k$ of partitions
by a standard Young tableau of a skew shape. We focus on partitions in $L(3,n)$.

\begin{enumerate}
 \item We first draw the $3\times n$ grid and color the Young diagram of $\mu_1$  by green/gray.

 \item For each $i\ge 1$, there is a unique square that is in the Young diagram of $\mu_{i+1}$ but not in the Young diagram of $\mu_{i}$. We fill
 the number $i$ in this square.
\end{enumerate}


Chain tableaux are helpful in discovering the patterns. Our map $\varphi$ produce chain decompositions of $L(3,n)$, each of the form
$C_\mu: \ \mu\lessdot \varphi(\mu)\lessdot \varphi^2(\mu)\lessdot \cdots\lessdot \varphi^s(\mu)$, with $\mu\in S_{3,n}$ and $\varphi^s(\mu)\in E_{3,n}$.
For sake of clarity, we list all the chain tableaux
of $L(3,8)$ in Figure \ref{fig:L38Tab}. For instance, $C_\phi$ corresponds to $\phi\lessdot 100 \lessdot 110\lessdot 111 \lessdot 211\lessdot\cdots \lessdot 888$,
where we have abbreviated the partition $(a,b,c)$ by $abc$ when clear from the context;
$C_{84}$ corresponds to the single partition $840$.

\begin{exam}
\begin{figure}[!ht]
\centering{
\includegraphics[height=3 in]{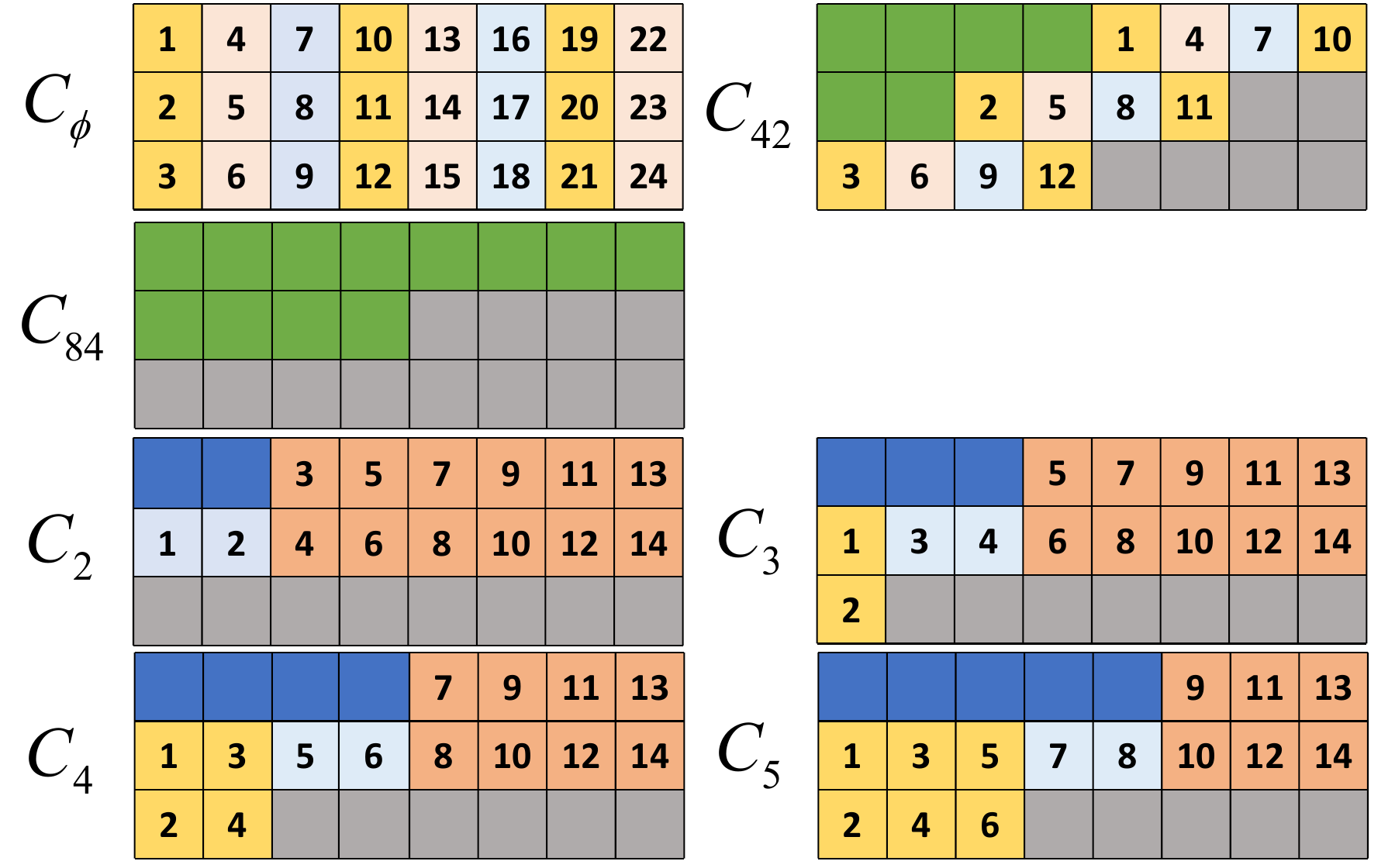}}
\end{figure}

\begin{figure}[!ht]
\centering{
\includegraphics[height=2.25 in]{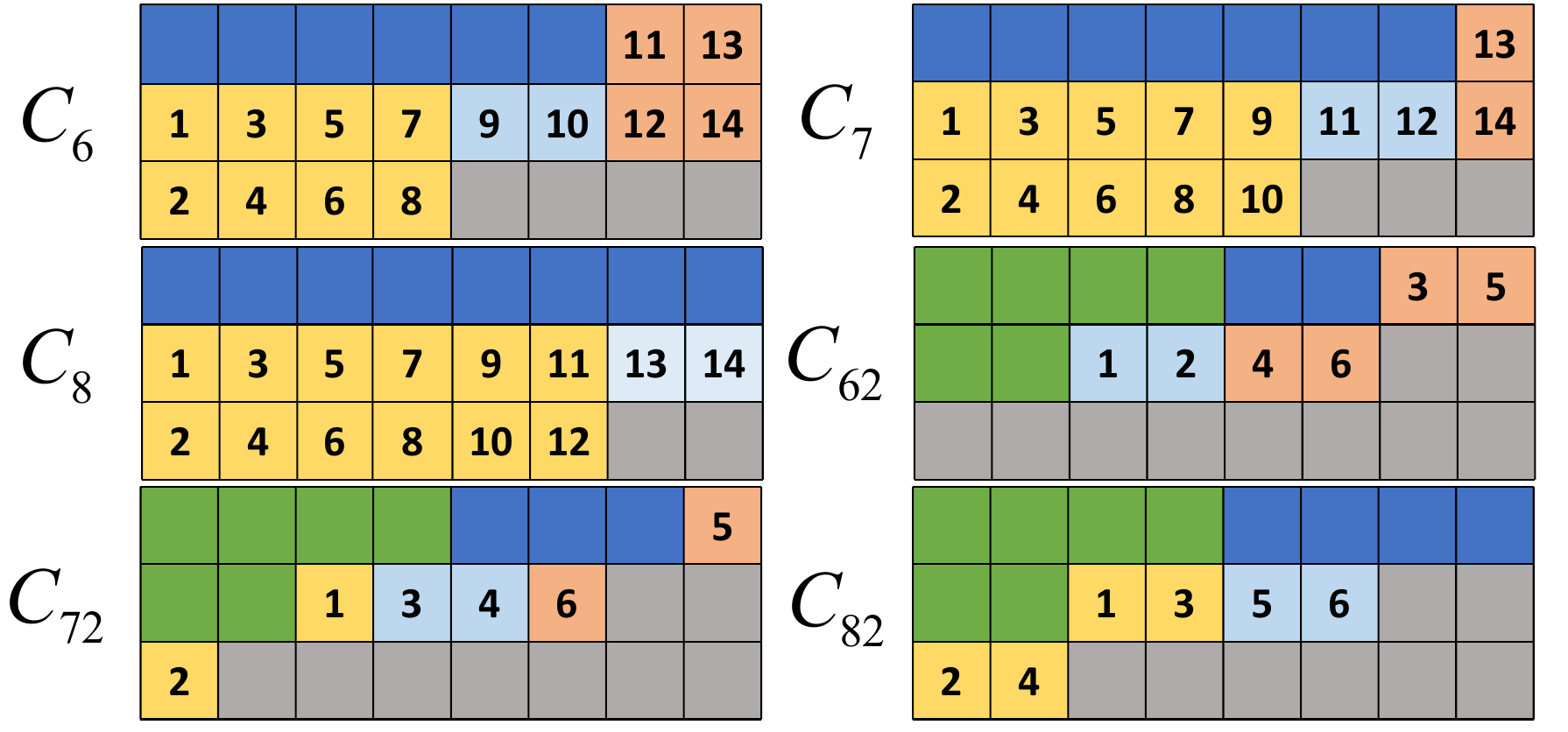}}
\caption{The tableaux of the chain decomposition of $L(3,8)$ under $\varphi$. \label{fig:L38Tab}}
\end{figure}
\end{exam}

By investigating these chain tableaux of $L(3,n)$ for small $n$, we find the pattern as follows.

\begin{theo}\label{theo:varphi-chain}
The bijection $\varphi$ in Theorem \ref{theo-varphi} produces a chain decomposition of $L(3,n)$.
The corresponding chain tableaux are divided into two types: i) $C_\mu$ with $\mu=(4k,2k,0)$ for $k=0,1,\dots, \lfloor n/4 \rfloor$, as in Figure
\ref{L3nType1};
ii) $C_\mu$ with $\mu=(4k+\ell,2k,0)$ where $\ell \ge 2$ and $k=0,1,\dots,\lfloor (n-\ell)/4 \rfloor$, as in Figure \ref{L3nType2}.
\begin{figure}[!ht]
\centering{
\includegraphics[height=1.6 in]{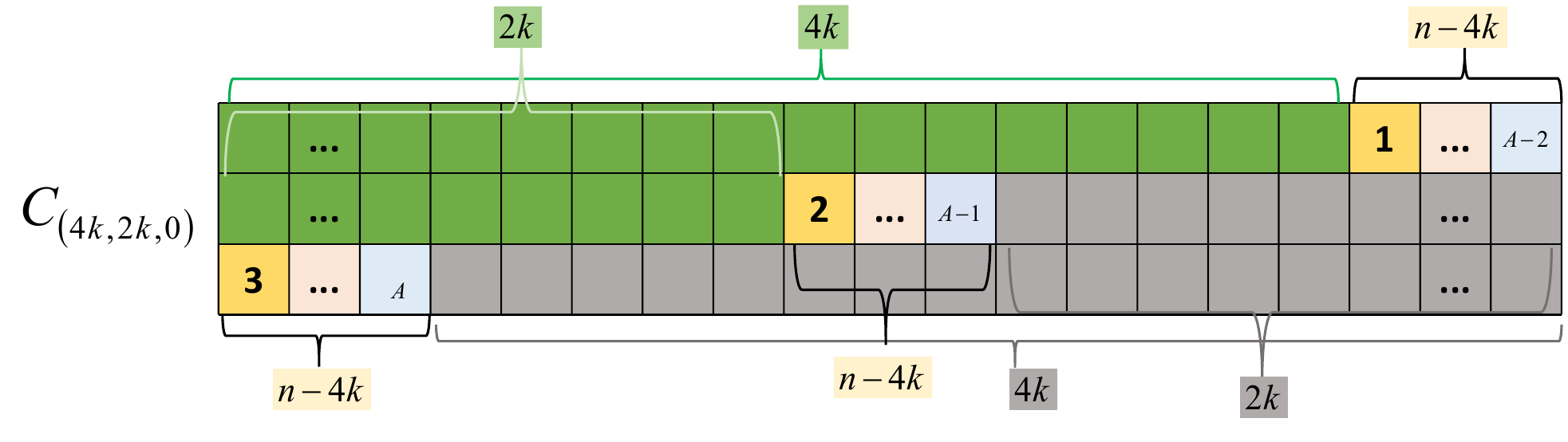}}
\caption{Chain tableaux of Type i) starting at $\mu=(4k,2k,0)$, where $A=3n-12k$. }\label{L3nType1}
\end{figure}

\begin{figure}[!ht]
\centering{
\includegraphics[height=1.6 in]{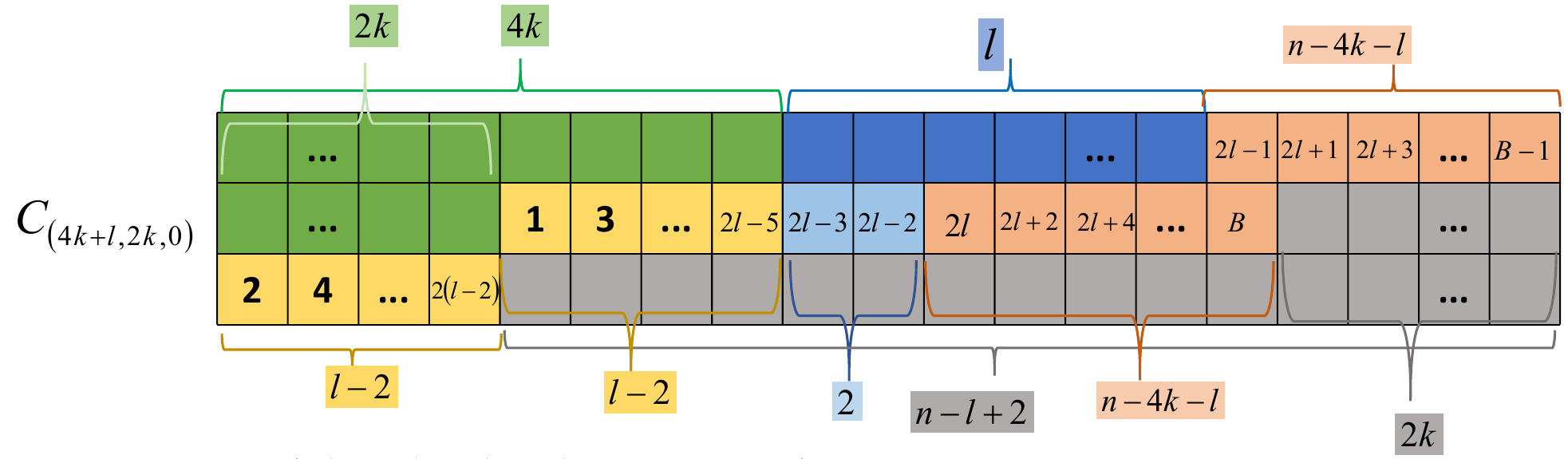}}
\caption{Chain tableaux of Type ii) starting at $\mu=(4k,2k,\ell)$ where $\ell\ge 2$ and $B=2n+4\ell-8k-2$.}\label{L3nType2}
\end{figure}
\end{theo}

We give two proofs of the theorem. The second proof will be given in the next section.
The first proof need the following Lemma.

\begin{lem}\label{lemProTab}
We have the following facts:
\begin{enumerate}
  \item[$(A_1)$] If $\lambda=(4k+c,2k+c,c),k\geqslant0,c\geqslant0$, then $\lambda \in E_{3,\lambda_1}$.
  \item[$(A_2)$] If $\lambda=(4k+c+1,2k+c,c),k\geqslant1,c\geqslant0$, then $\lambda \notin E_{3,\lambda_1}$,$\lambda_2+\lambda_3\equiv0\pmod 2$ and $(\lambda_1-1,\lambda_2+1,\lambda_3)\notin E_{3,\lambda_1-1}$.
  \item[$(A_3)$] If $\lambda=(4k+c+1,2k+c+1,c),k\geqslant1 ,c\geqslant0$, then $\lambda \notin E_{3,\lambda_1}$, $\lambda_2+\lambda_3\not\equiv0\pmod2$ and $(\lambda_1-1,\lambda_2,\lambda_3+1)\notin E_{3,\lambda_1-1}$.
  \item[$(B_2)$] If $\lambda=(4k+\ell,2k+c,c),k\geqslant0 ,1\leqslant c \leqslant \ell-2,\ell\geqslant2$, then $\lambda \notin E_{3,\lambda_1}$,$\lambda_2+\lambda_3\equiv0\pmod 2$ and $(\lambda_1-1,\lambda_2+1,\lambda_3)\notin E_{3,\lambda_1-1}$.
  \item[$(B_3)$] If $\lambda=(4k+\ell,2k+c+1,c),k\geqslant0 ,1\leqslant c \leqslant \ell-3,\ell\geqslant2$, then $\lambda \notin E_{3,\lambda_1}$, $\lambda_2+\lambda_3\not\equiv0\pmod2$ and $(\lambda_1-1,\lambda_2,\lambda_3+1)\notin E_{3,\lambda_1-1}$.
  \item[$(C_1)$] If $\lambda=(4k+\ell+c,2k+\ell+c,\ell-2),c\geqslant 0,k\geqslant0,\ell\geqslant2$, then $\lambda \in E_{3,\lambda_1}$.
  \item[$(C_2)$] If $\lambda=(4k+\ell+c+1,2k+\ell+c,\ell-2),k\geqslant0 ,c\geqslant0,\ell\geqslant2$, then $\lambda \notin E_{3,\lambda_1}$ and  $\lambda$ satisfy one of the following two conditions:\\
  (1) $\lambda_2+\lambda_3\equiv0\pmod 2$ and $(\lambda_1-1,\lambda_2+1,\lambda_3)\notin E_{3,\lambda_1-1}$;\\
  (2) $\lambda_2+\lambda_3\not\equiv0\pmod 2$ and $(\lambda_1-1,\lambda_2,\lambda_3+1)\in E_{3,\lambda_1-1}$.
\end{enumerate}
\end{lem}
\begin{proof}
\begin{enumerate}
  \item[$(A_1)$]  Direct calculation gives $\lambda_1-\lambda_2=2k$ and $2\lambda_2-\lambda_1-\lambda_3=0\in \None$.

  \item[$(A_2)$]  By direct calculation, $\lambda_1-\lambda_2=2k+1$ is odd, which implies $\lambda \notin E_{3,\lambda_1}$; $\lambda_2+\lambda_3=2k+2c$ is even;
  Let $\rho=(\lambda_1-1,\lambda_2+1,\lambda_3)$. Then $\rho_1-\rho_2=2k-1$ is odd, which implies that  $\rho \notin E_{3,\lambda_1-1}$.

  \item[$(A_3)$]  By direct calculation, $2\lambda_2-\lambda_1-\lambda_3=1 \notin \None$, which implies $\lambda \notin E_{3,\lambda_1}$; $\lambda_2+\lambda_3=2k+2c+1$ is odd;
   Let $\rho=(\lambda_1-1,\lambda_2,\lambda_3+1)$. Then $\rho_1-\rho_2=2k-1$ implies $\rho \notin E_{3,\lambda_1-1}$.

  \item[$(B_2)$] By $2\lambda_2-\lambda_1-\lambda_3=-(\ell-c)(\le -2) \notin \None$ we have $\lambda \notin E_{3,\lambda_1}$; $\lambda_2+\lambda_3=2k+2c$ is even;
  Let $\rho=(\lambda_1-1,\lambda_2+1,\lambda_3)$. Then $2\rho_2-\rho_1-\rho_3=-(\ell-c)\le-2$ and is hence not in $\None$, so that $\rho \notin E_{3,\lambda_1-1}$.

  \item[$(B_3)$] By $2\lambda_2-\lambda_1-\lambda_3=-(\ell-c)+2(\le -1) \notin \None$ we have $\lambda \notin E_{3,\lambda_1}$; $\lambda_2+\lambda_3=2k+2c+1$ is odd;
  Let $\rho=(\lambda_1-1,\lambda_2,\lambda_3+1)$. Then $2\rho_2-\rho_1-\rho_3=-(\ell-c)+1(\le-2) \notin \None$, so that $\rho \notin E_{3,\lambda_1-1}$.

  \item[$(C_1)$] Since $\lambda_1-\lambda_2=2k$ is even and $2\lambda_2-\lambda_1-\lambda_3=2c+2 \in \None$, we have $\lambda \in E_{3,\lambda_1}$.

  \item[$(C_2)$] By $\lambda_1-\lambda_2=2k+1$ we have $\lambda \notin E_{3,\lambda_1}$.

                      (1) If $\lambda_2+\lambda_3=2k+2\ell+c-2$ is even. Let $\rho=(\lambda_1-1,\lambda_2+1,\lambda_3)=(4k+\ell+c,2k+\ell+c+1,\ell-2)$.
                      Then $\rho_1-\rho_2=2k-1$ is odd, so that $\rho \notin E_{3,\lambda_1-1}$.

                      (2) If $\lambda_2+\lambda_3=2k+2\ell+c-2$ is odd. Let $\rho=(\lambda_1-1,\lambda_2,\lambda_3+1)=(4k+\ell+c,2k+\ell+c,\ell-1)$.
                      Then $\rho_1-\rho_2=0$ is even. Since $\lambda_2+\lambda_3=2k+2\ell+c-2$ is odd, we have $c\neq 0$ and therefore $2\rho_2-\rho_1-\rho_3=c+1 \in \None$.
  \end{enumerate}

\end{proof}

\begin{proof}[First proof of Theorem \ref{theo:varphi-chain}]
We prove by applying Lemma
\ref{lemProTab} and the definition of $\varphi$. There are two cases as follows.
\begin{enumerate}
                 \item[Case (i).] See Figure \ref{L3nType1}.
                                 For each $k,c\in \N$, let $\mu_c=(\mu_1,\mu_2,\mu_3)=(4k+c,2k+c,c)$.
                                 By Lemma \ref{lemProTab} part $(A_1)$, $(A_2)$ and $(A_3)$, we get $\varphi(\mu_c)=(\mu_1+1,\mu_2,\mu_3)=(4k+c+1,2k+c,c)$, $\varphi^2(\mu_c)=(\mu_1+1,\mu_2+1,\mu_3)=(4k+c+1,2k+c+1,c)$ and $\varphi^3(\mu_c)=(\mu_1+1,\mu_2+1,\mu_3+1)=(4k+c+1,2k+c+1,c+1)$ respectively.

               \item[Case (ii).] See Figure \ref{L3nType2}. First consider labels up to $2\ell -5$.
                                 For each $k\in \N$ and $\ell\geqslant2$ satisfying $1\leqslant c \leqslant \ell-3$, let $\nu_{\ell,c}=(4k+\ell,2k+c,c)$.
By part $(B_2)$ and $(B_3)$, we have
$$\nu_{\ell,c}\mathop{\longrightarrow}\limits^\varphi (4k+\ell,2k+c+1,c)\mathop{\longrightarrow}\limits^\varphi(4k+\ell,2k+c+1,c+1)=\nu_{\ell,c+1}.$$
This process end at $\omega=\nu_{\ell,\ell-2}=(4k+\ell,2k+\ell-2,\ell-2)$ with label $2\ell - 4$.

Next consider $\omega$ with label $2\ell-4$ and
$\rho=(4k+\ell,2k+\ell-1,\ell-2)$ with label $2\ell -3$. By part $(B_2)$, we get $\varphi(\omega)=\rho$, as desired.
We need to show that $\rho\in F_2^o$ so that $\varphi(\rho)=(4k+\ell,2k+\ell,\ell-2)$ DO corresponds to label $2\ell -2$. Firstly $\rho_2+\rho_3=2k+2\ell-3$ is odd; Secondly let $\tau=(\rho_1-1,\rho_2,\rho_3+1)$. Then
$\tau_1-\tau_2=2k$ is even and $2\tau_2-\tau_1-\tau_3=0 \in \None$, so that $\tau\in E_{3,\rho_1-1}$; Finally $\rho_1-\rho_2=2k+1$ is odd, which implies that $\rho\notin E_{3,\rho_1}$.

For labels begin at $2\ell -2$, let $\mu_{\ell,c}=(4k+\ell+c,2k+\ell+c,\ell-2)$, where $k\in \N$ and $\ell\geqslant2$ satisfy $0\leqslant c \leqslant n-4k-\ell-1$. By part $(C_1)$ and $(C_2)$,
we have
\begin{multline*}
 \qquad  \mu_{\ell,c}\mathop{\longrightarrow}\limits^\varphi (4k+\ell+c+1,2k+\ell +c,\ell-2)\\
\mathop{\longrightarrow}\limits^\varphi(4k+\ell+c+1,2k+\ell +c+1,\ell-2)=\mu_{\ell,c+1}. \qquad
\end{multline*}
This process ends at $\mu_{\ell, n-4k-\ell}=(n,n-2k,\ell-2)$.
\end{enumerate}
\end{proof}

From the proof, we see that $\varphi$ induces two type of chains: i) Chains from $(4k,2k,0)$ to $(n,n-2k,n-4k)=(4k,2k,0)^*$;
ii) For $\ell \ge 2$, we have chains from $(4k+\ell,2k,0)$ to $(n,n-2k,\ell-2)=(n-\ell+2, 2k,0)^*$. This give rise an involution $\psi$ on $S_{3,n}$
defined by $\psi((4k,2k,0))=(4k,2k,0)$ and $\psi((4k+\ell,2k,0))=(n-\ell+2, 2k,0)$     for $\ell\ge 2$. The fixed points of $\psi$ are
$\{(4k,2k,0): 4k\le n\} \cup \{(4k+\ell,2k,0): \ell\ge 2, n=4k+2\ell-2\}$.

\subsection{Comparison with other chain decompositions}

In Figure \ref{L36B} we draw the tableaux of the symmetric chain decompositions of $L(3,6)$ from a result of Bernt Lindstr$\ddot{o}$m \cite{B.Lindstrom}. Compare it with our chain decompositions in Figure \ref{L36our}. We also draw the tableaux of the symmetric chain  decompositions of $L(4,4)$ from a result of Douglas B. West \cite{Nathan} in Figure \ref{L44D}. Compare it with our chain  decompositions in Figure \ref{L44our}. In both examples, the pattern of our chain decompositions seems easier to find.

\begin{figure}[!ht]
\centering{
\includegraphics[height=2.4 in]{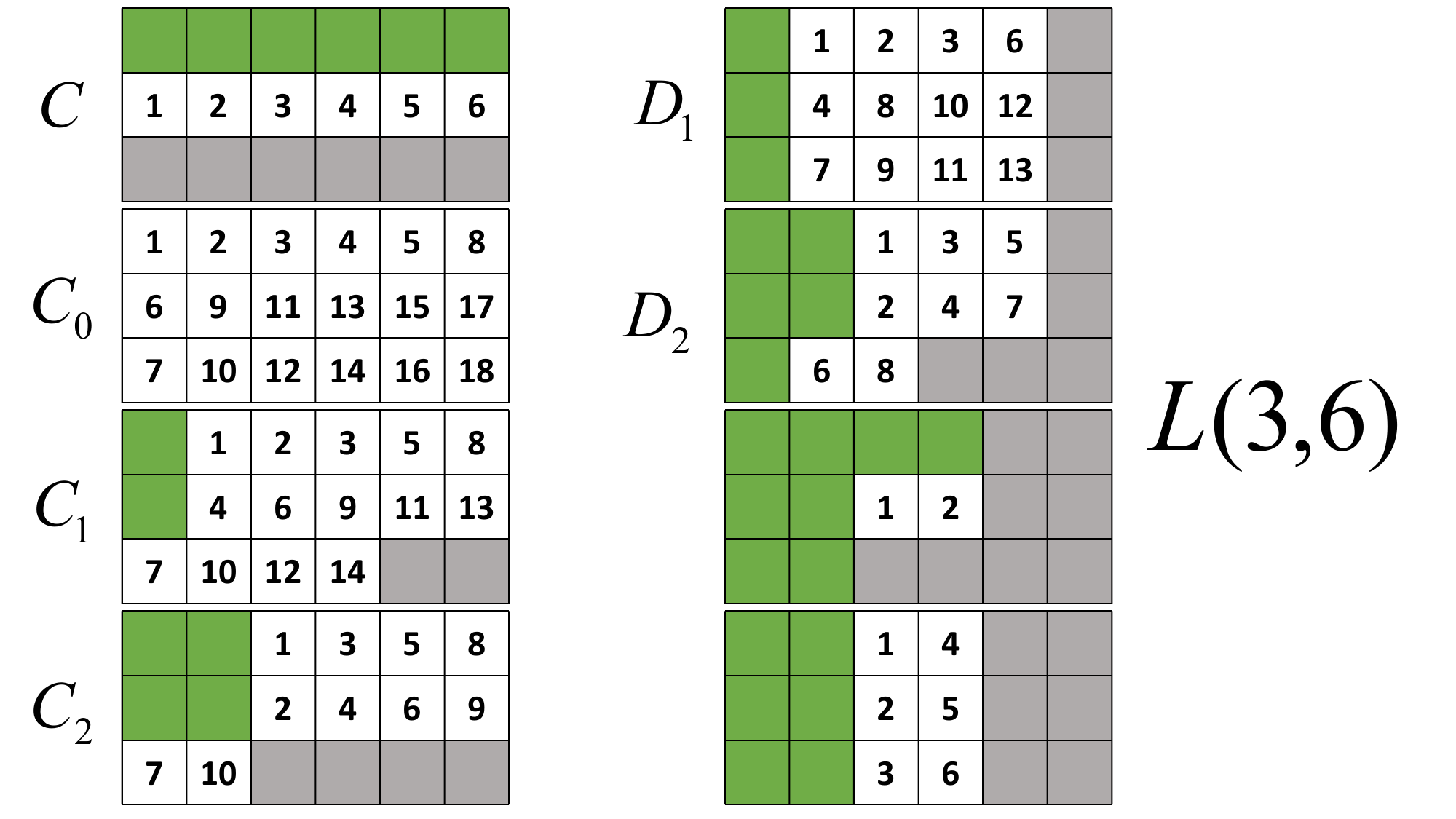}}
\caption{The tableaux for Lindstr$\ddot{o}$m's chain decompositions of $L(3,6)$. }\label{L36B}
\end{figure}

\begin{figure}[!ht]
\centering{
\includegraphics[height=2.4 in]{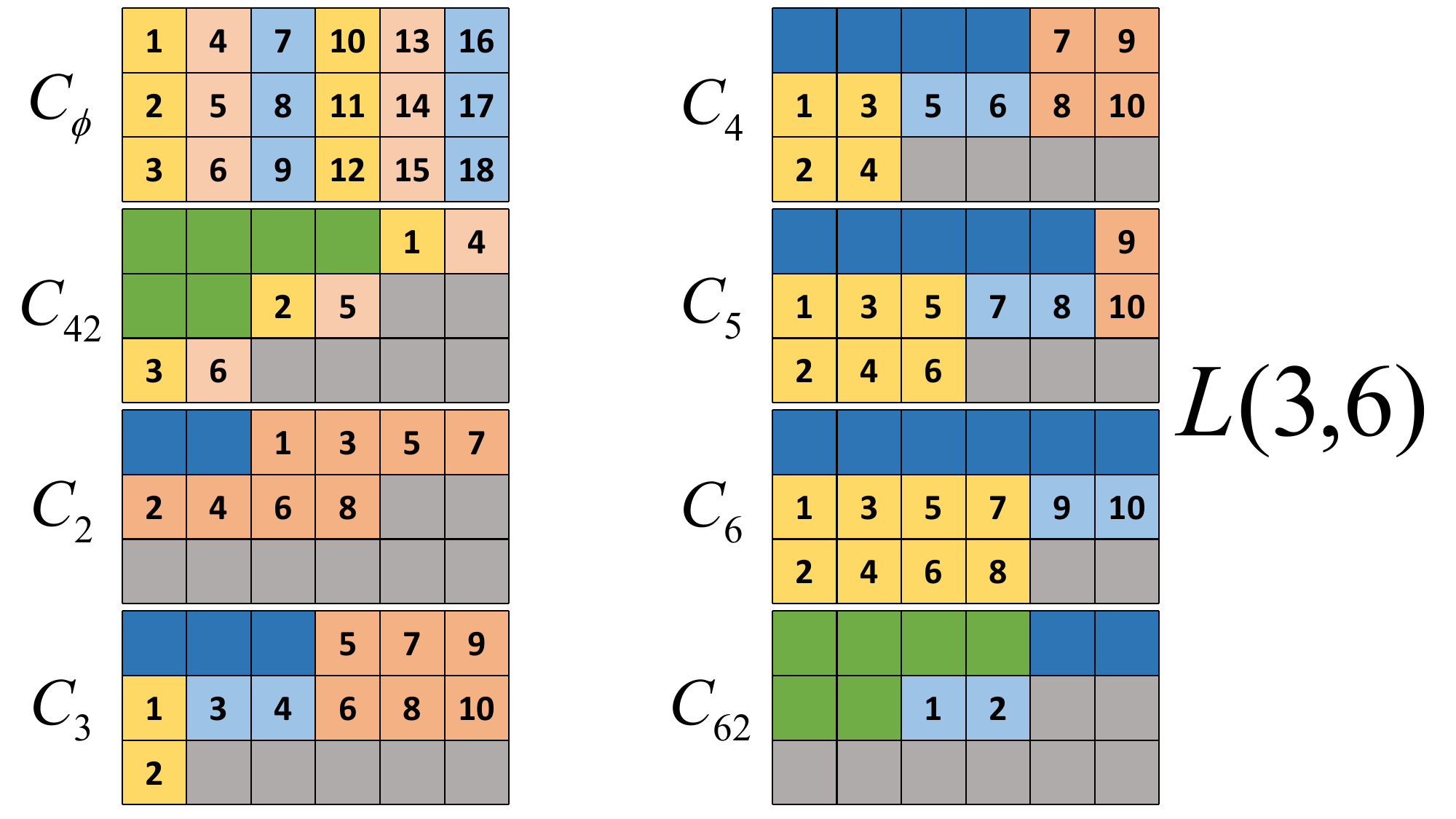}}
\caption{The tableaux for our chain decompositions of $L(3,6)$. }\label{L36our}
\end{figure}

\begin{figure}[!ht]
\centering{
\includegraphics[height=2.4 in]{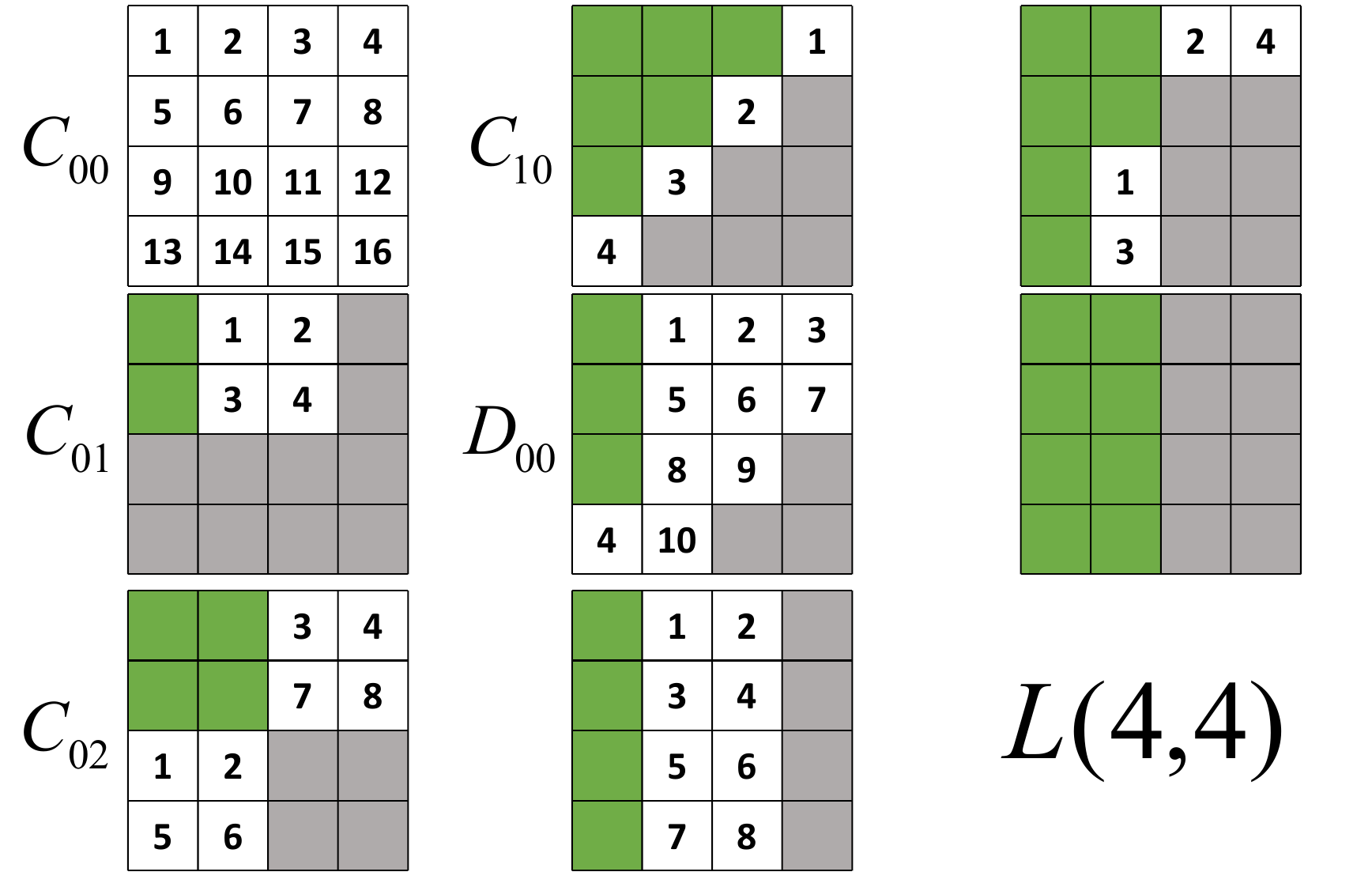}}
\caption{The tableaux for West's chain decompositions of $L(4,4)$.}\label{L44D}
\end{figure}

\begin{figure}[!ht]
\centering{
\includegraphics[height=2.4 in]{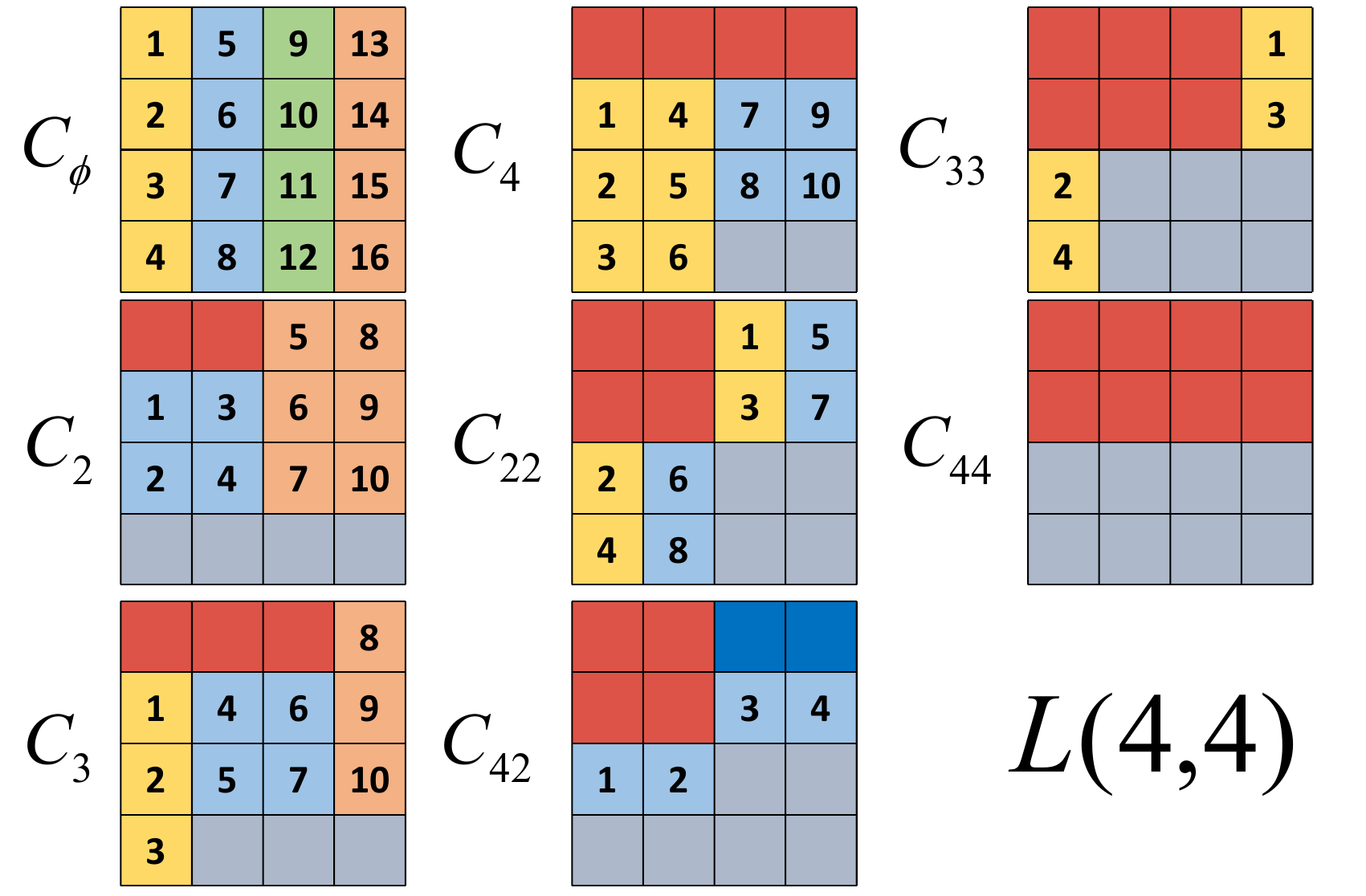}}
\caption{The tableaux for our chain decompositions of $L(4,4)$.}\label{L44our}
\end{figure}

\section{Sperner chain decompositions of $L(3,n)$ and $L(4,n)$ \label{sec:L3nL4n}}
A chain decomposition $C_1,C_2,\dots, C_N$ of $L(m,n)$ is called Sperner if it satisfies the following two conditions:
i) $L(m,n)$ is the disjoint union of the $C_j$'s;
ii) each $C_j$ is of the form $x_{j,1} < x_{j,2} < x_{j,3} <\cdots < x_{j,e_j}$
with $\rank(x_{j,1})\le mn/2$ and $\rank(x_{j,e_j})\ge mn/2$, where $x_{j,1}$ are called the starting partitions and $x_{j,e_j}$ are called the end partitions.
Obviously each $C_j$ intersects $L_{\lfloor mn/2 \rfloor}(m,n)$, which implies the Sperner property.
Symmetric chain decompositions are Sperner chain decompositions satisfying the extra rank symmetric condition $\rank(x_{j,1})+\rank(x_{j,e_j})=mn$.

Theorem \ref{theo:varphi-chain} indeed give a Sperner chain decomposition of $L(3,n)$. Its first proof relies on the order matching $\varphi$.
We give a self-contained proof and extend the result for $L(4,n)$.

\subsection{A direct proof for the chain decomposition of $L(3,n)$}
We need the following classification of $L(3,n)$ in 7 types.
\begin{lem}\label{lemTabEle}
Any element $\lambda=(\lambda_1,\lambda_2,\lambda_3)$ in $L(3,n)$ can be uniquely expressed in one of the following forms:
\begin{enumerate}
  \item[$A_1)$] $(4k+c,2k+c,c),\ k\geqslant0,\ c\geqslant0$.
  \item[$A_2)$] $(4k+c+1,2k+c,c),\ k\geqslant0,\ c\geqslant0$.
  \item[$A_3)$] $(4k+c+1,2k+c+1,c), \ k\geqslant0,\ c\geqslant0$.
  \item[$B_2)$] $(4k+\ell,2k+c,c),\ k\geqslant0,\ \ell\geqslant2,\ \ell-2\geqslant c \geqslant0$.
  \item[$B_3)$] $(4k+\ell,2k+c+1,c),\ k\geqslant0,\ \ell\geqslant2,\ \ell-2\geqslant c \geqslant0$.
  \item[$C_1)$] $(4k+\ell+c,2k+\ell+c,\ell-2),\ k\geqslant0,\ \ell\geqslant2,\ c\geqslant0$.
  \item[$C_2)$] $(4k+\ell+c+1,2k+\ell+c,\ell-2),\ k\geqslant0,\ \ell\geqslant2,\ c\geqslant0$.
\end{enumerate}
\end{lem}
\begin{proof}
We first prove the uniqueness.
Since the elements in each type are clearly different from each other, it suffices to prove
that there are no identical elements between different types. This is achieved by computing the values $\alpha=\lambda_1-\lambda_2$,
$\beta=\lambda_2-\lambda_3$, and $\alpha-\beta$ for each $\lambda$, as given in the following table.
$$\begin{array}{|l|l|l|l|l|}\hline
   &     (\lambda_1,\lambda_2,\lambda_3)      & \alpha-\beta   & \alpha=\lambda_1-\lambda_2   & \beta=\lambda_2-\lambda_3 \\\hline
A_1)&     (4k+c,2k+c,c)                        & 0                                &  even                 &     even                \\\hline
A_2)&     (4k+c+1,2k+c,c)                      & 1                                &  odd                  &     even               \\\hline
A_3)&     (4k+c+1,2k+c+1,c)                    & -1                               &  even                 &     odd                \\\hline
B_2)&     (4k+\ell,2k+c,c)                     & \ell -c \geq 2                   &                       &     even               \\\hline
B_3)&     (4k+\ell,2k+c+1,c)                   & \ell -c-2 \geq 0                 &                       &     odd                \\\hline
C_1)&     (4k+\ell+c,2k+\ell+c,\ell-2)         & -2-c \leq -2                     &  even                 &                       \\\hline
C_2)&     (4k+\ell+c+1,2k+\ell+c,\ell-2)       & -1-c \leq -1                     &  odd                  &\\\hline
\end{array}
$$
For instance, type $A_1$ and $B_3$ partitions can only overlap at $\lambda$ with $\alpha-\beta=0$, but their $\beta$ values have different parity. The other cases
can be done similarly.

Next we prove that any element in $L(3,n)$ belongs to one of the seven types. Let $\alpha$ and $\beta$ be defined as above for a given $\lambda=(\lambda_1,\lambda_2,\lambda_3)$.
The following table determines the type of $\lambda$ and their corresponding representations.

$$\begin{array}{|l|l|l|l|l|l|l|}\hline
   \alpha-\beta         & \alpha       &\beta   &k                                 &c                         &\ell                      &\text{type of } \lambda\\\hline
  0                     &  even        & even   &\frac{\alpha}{2}                  &\lambda_3                 & \times                   &A_1)\\\hline
  1                     &  odd         & even   &\frac{\beta}{2}                   &\lambda_3                 & \times                   &A_2)\\\hline
 -1                     &  even        & odd    &\frac{\alpha}{2}                 &\lambda_3                  & \times                   &A_3)\\\hline
 \geq 2                 &              & even   &\frac{\beta}{2}                   &\lambda_3                 &(\alpha-\beta)+\lambda_3  &B_2)\\\hline
 \geq 0                 &              & odd    &\frac{\beta-1}{2}                 &\lambda_3                &(\alpha-\beta)+\lambda_3+2 &B_3)\\\hline
 \leq -2                &  even        &        &\frac{\alpha}{2}                  &-(\alpha-\beta)-2         &\lambda_3+2               &C_1)\\\hline
 \leq -1                &  odd         &        &\frac{\alpha-1}{2}                &-(\alpha-\beta)-1         &\lambda_3+2               &C_2)\\ \hline
\end{array}.
$$
This completes the proof.
\end{proof}

\begin{proof}[Second proof of Theorem \ref{theo:varphi-chain}]
The theorem clearly follows by the following Claims 1 and 2.

Claim 1: For each $k$, $C_\mu$ with $\mu=(4k,2k,0)$ contains elements of type $A$ for all $c$.

Starting at the $A_1$ element $(4k,2k,0)$ with $c=0$, we successively
add $1$ to the first row, the second row and the third row to get $A_2$, $A_3$, and $A_1$ elements respectively. Now we are at the $A_1$ element $(4k+1,2k+1,1)$ with $c=1$.
Continuing this way, we see that Claim 1 holds true.

Claim 2: For each $k$ and $\ell \ge 2$, $C_\mu$ with $\mu=(4k+\ell,2k,0)$ contains all elements of type $B$ and $C$ for all $c$.

Starting at $B_2$ element $(4k+\ell,2k,0)$ with $c=0$, we successively add $1$ to the second row, and the third row to get $B_3$ and $B_2$
elements respectively. Now we are at the $B_2$ element $(4k+\ell,2k+1,1)$ with $c=1$. Continuing this way until we reach
the $B_2$ element $(4k+\ell,2k+\ell-2,\ell -2)$ with $c=\ell-2$. By adding $1$ to the second row, we get the $B_3$ element $(4k+\ell,2k+\ell-1,\ell -2)$
with $c=\ell-2$. This covers all type $B$ elements.

Next we add $1$ to the second row to get the $C_1$ element $(4k+\ell,2k+\ell,\ell -2)$ with $c=0$. After that, we successively add $1$ to the first row, and
the second row to get $C_2$ and $C_1$ elements respectively. Now we are at the $C_1$ element $(4k+\ell+1,2k+\ell+1,\ell -2)$ with $c=1$. Continuing this way, we see that Claim 2 holds true.
\end{proof}

\subsection{Sperner chain decompositions of $L(4,n)$}
Our chain decompositions for $L(4,n)$ are also from the greedy algorithm. Let us see Figure \ref{fig:L48Tab} for the Chain tableaux of $L(4,8)$ for the pattern.
\begin{figure}[!ht]
\centering{
\includegraphics[height=2.4 in]{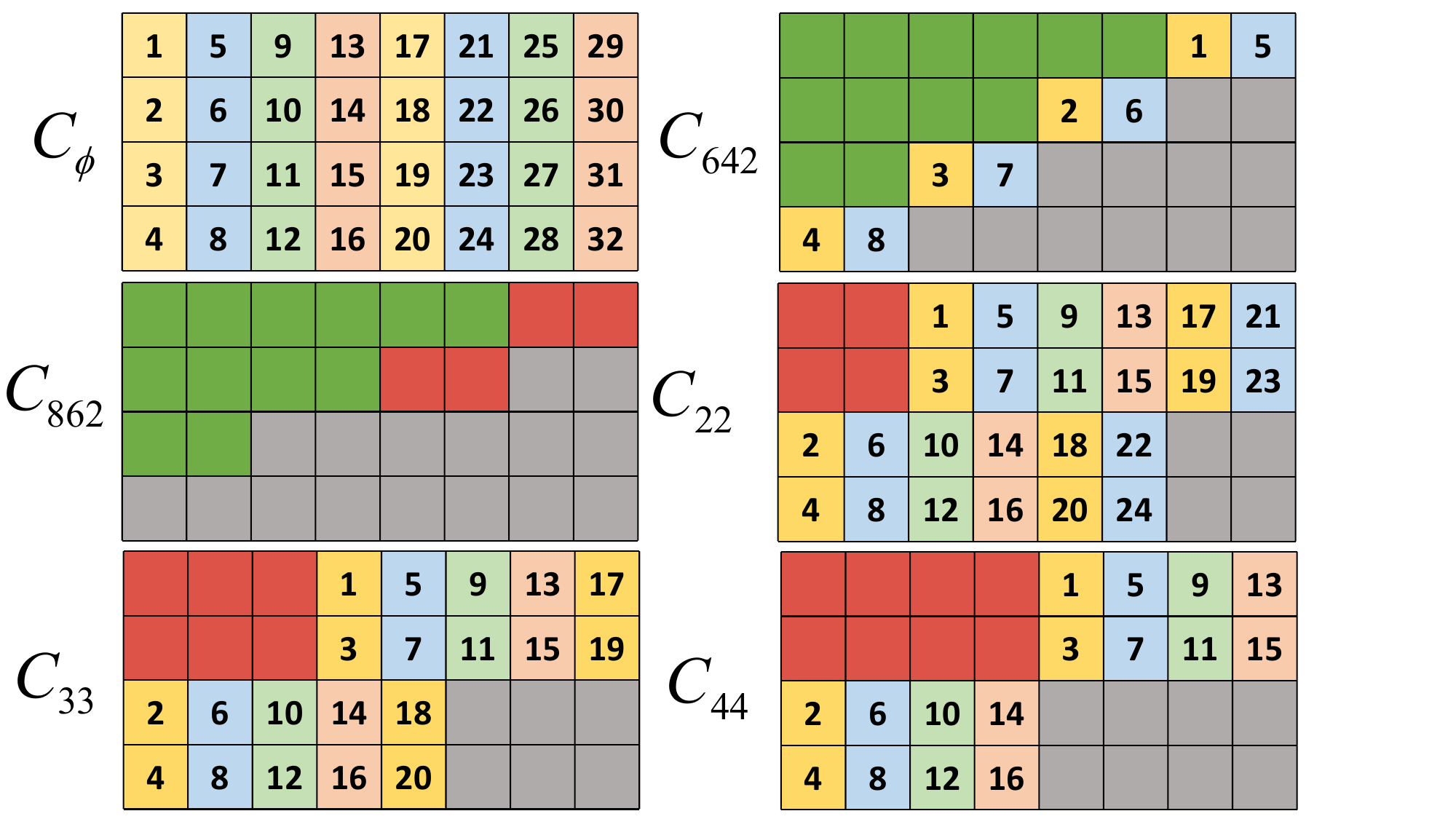}}
\end{figure}

\begin{figure}[!ht]
\centering{
\includegraphics[height=2.4 in]{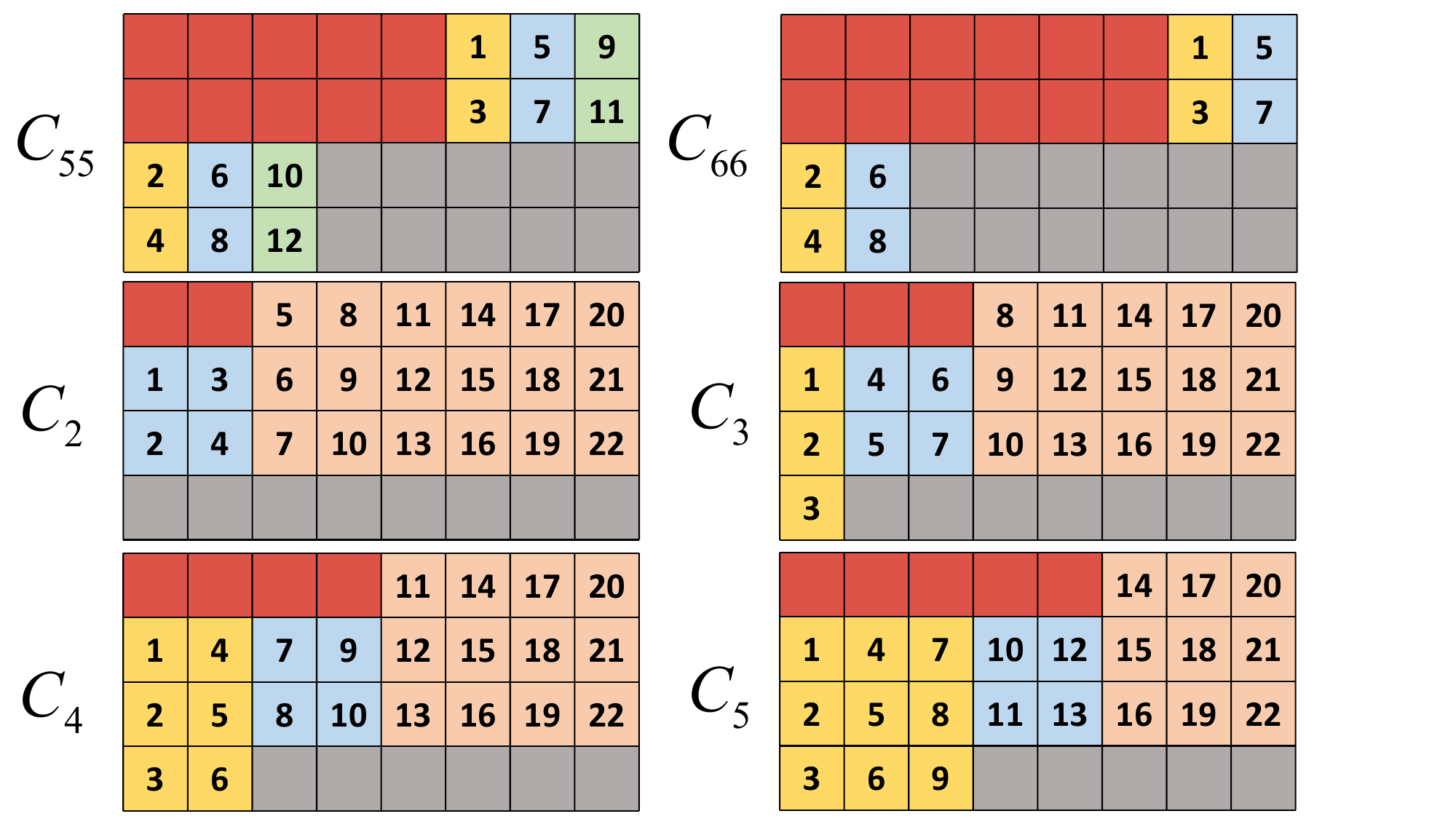}}
\end{figure}

\begin{figure}[!ht]
\centering{
\includegraphics[height=2.4 in]{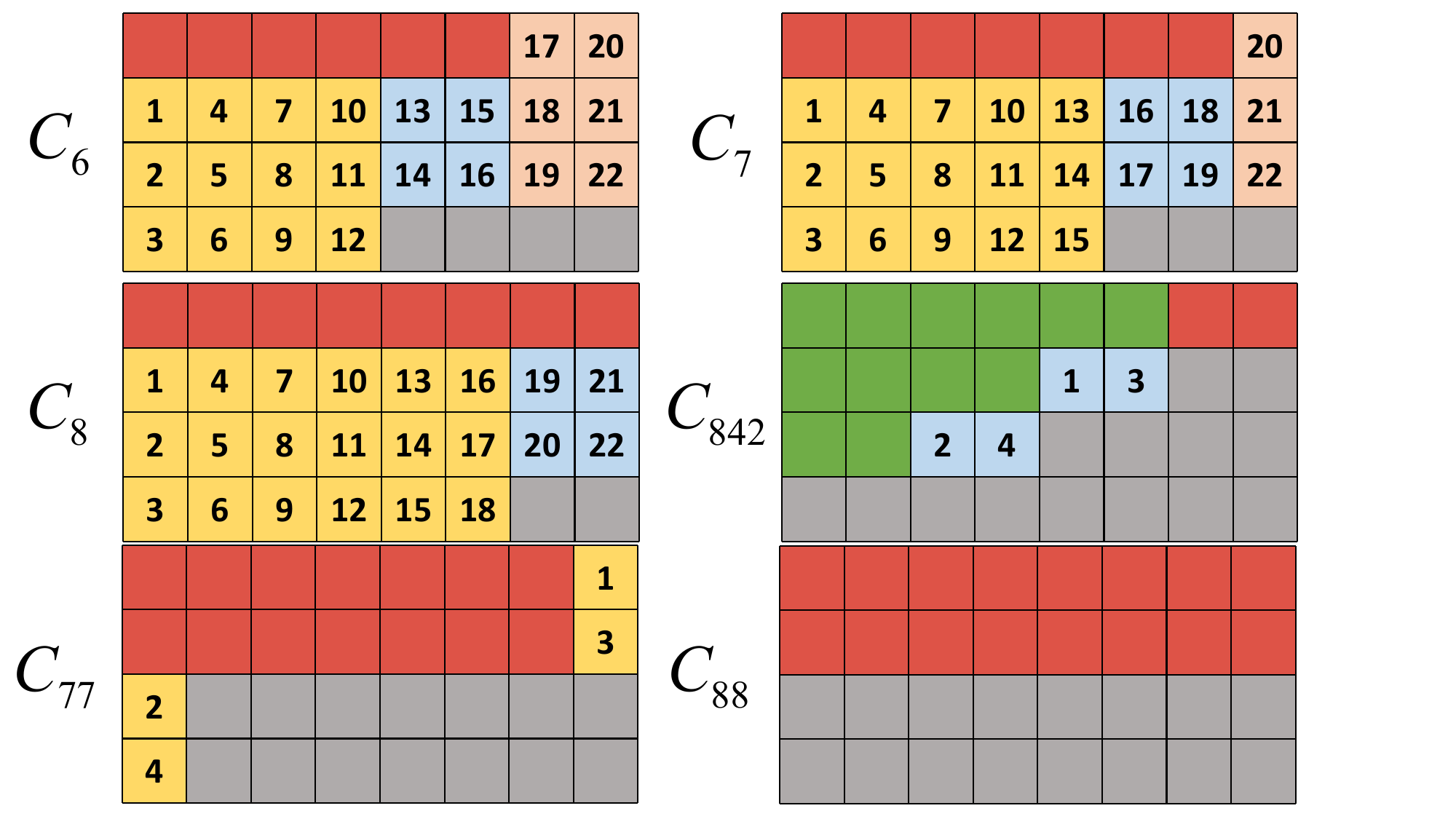}}
\end{figure}

\begin{figure}[!ht]
\centering{
\includegraphics[height=2.4 in]{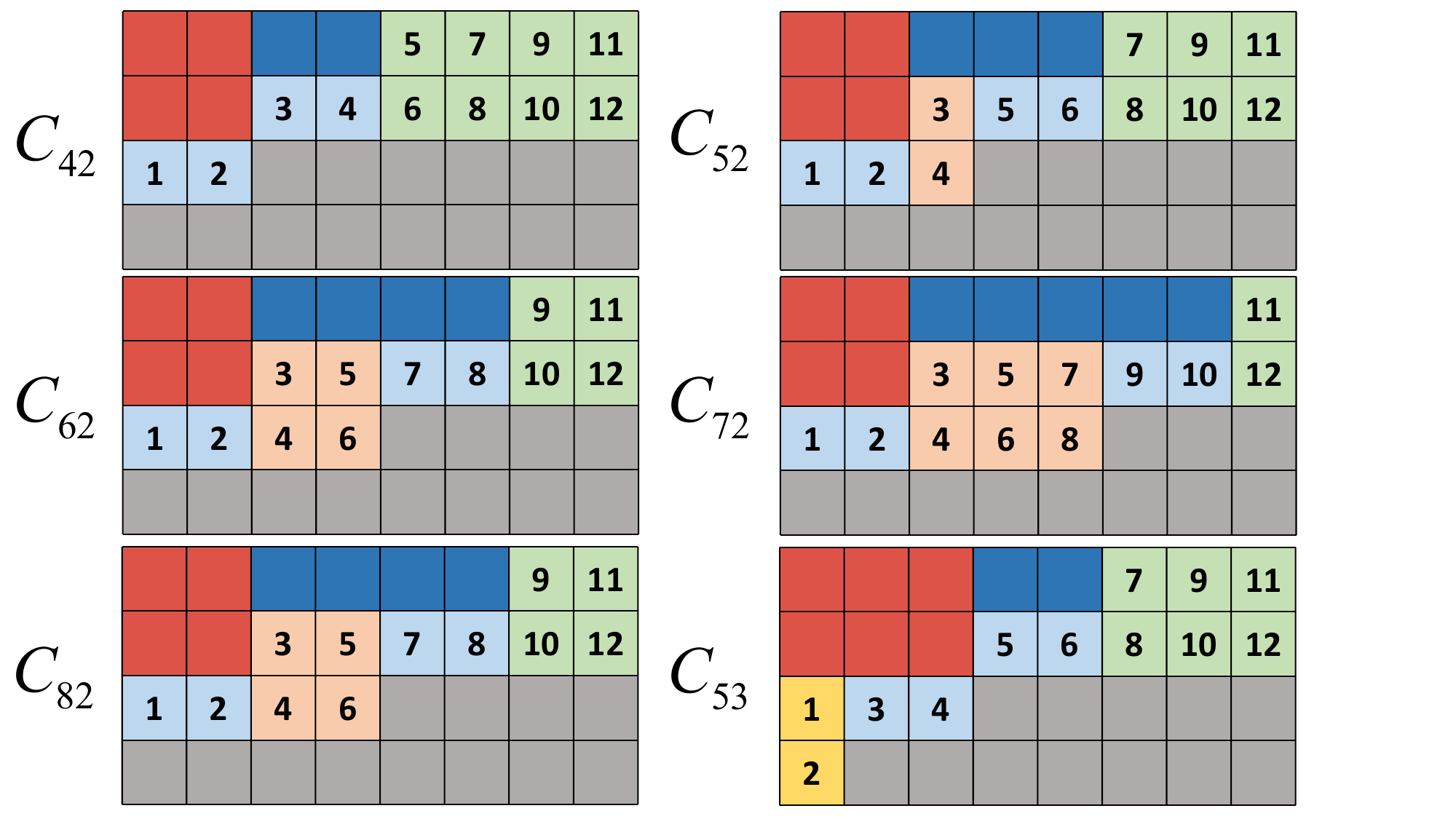}}
\end{figure}

\begin{figure}[!ht]
\centering{
\includegraphics[height=2.4 in]{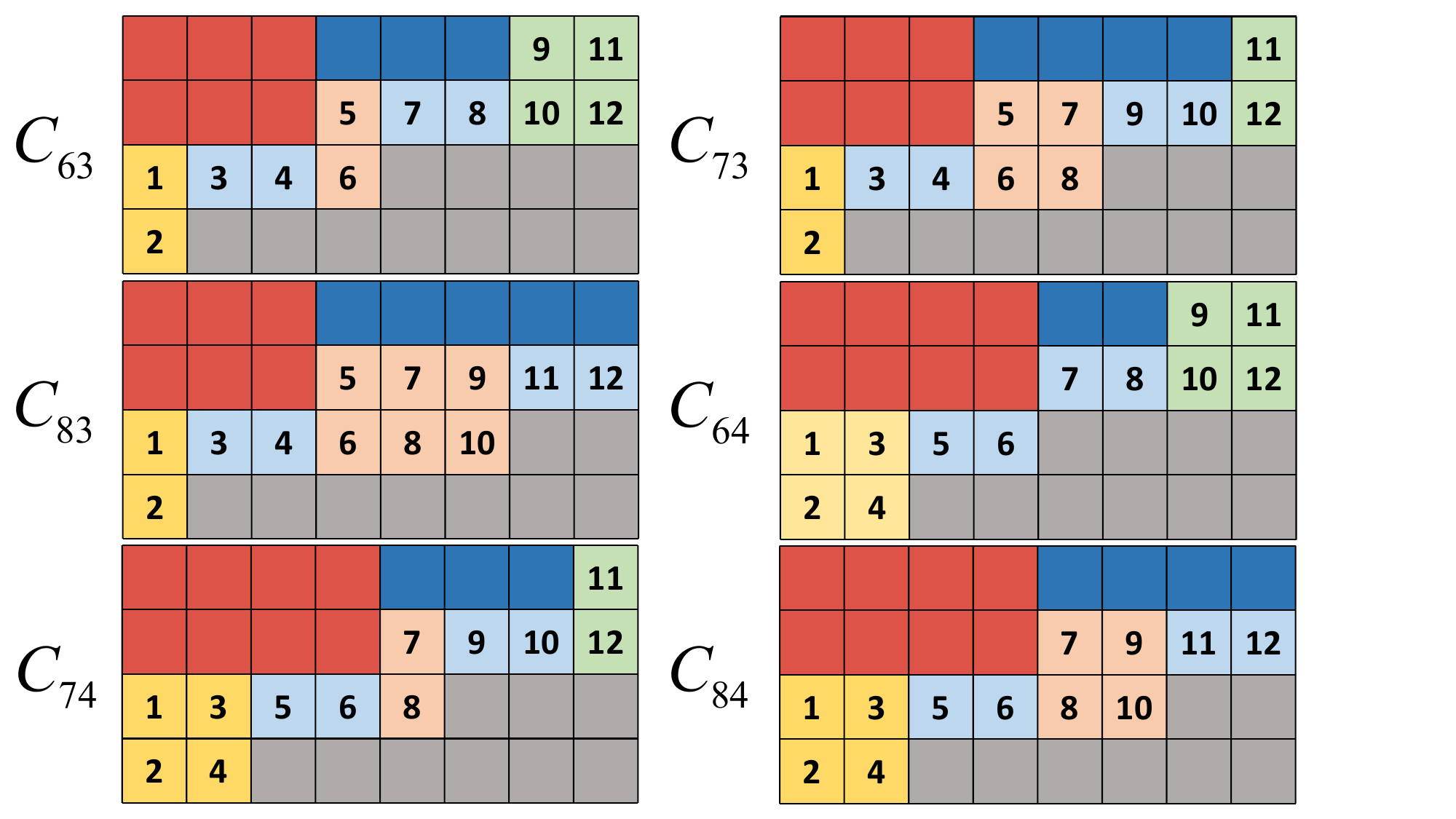}}
\end{figure}

\begin{figure}[!ht]
\centering{
\includegraphics[height=1.64 in]{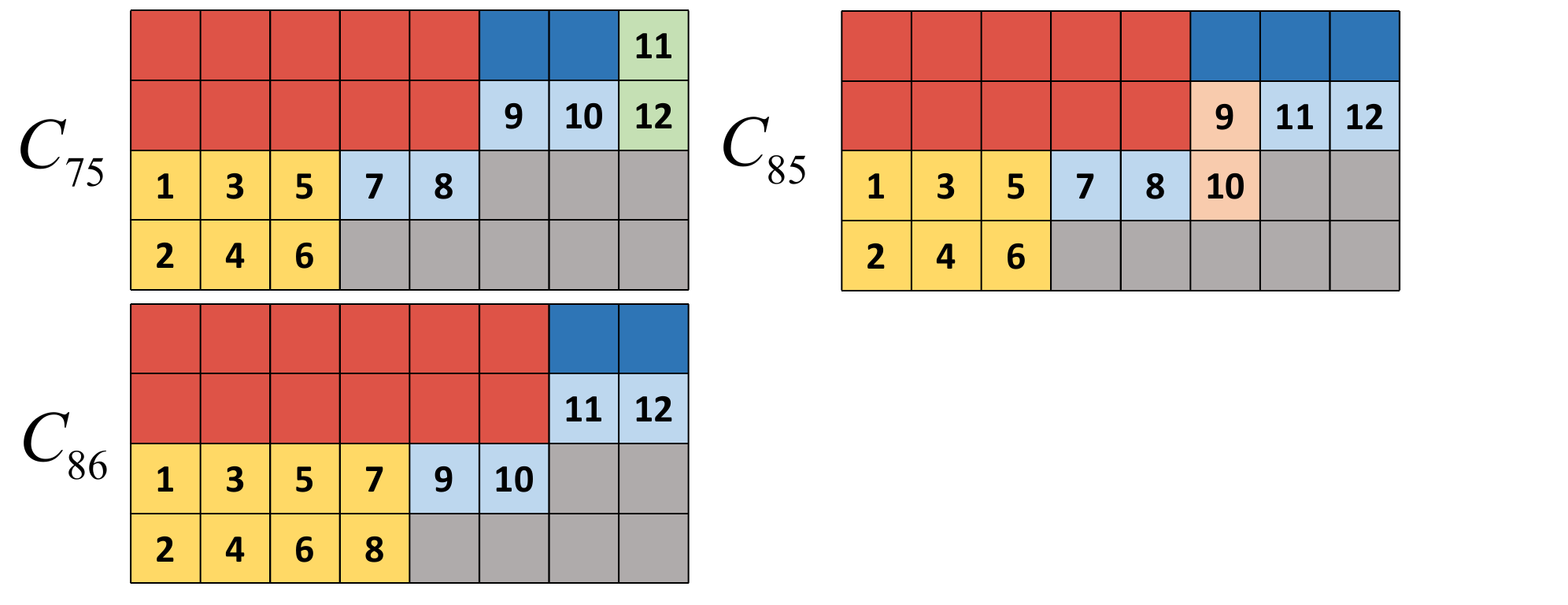}}
\caption{The tableaux of all $L(4,8)$ chains. \label{fig:L48Tab}}
\end{figure}

For general $n$, the result is summarized as follows.
\begin{theo}\label{theo:L4n-chain-dec}
The Young's lattice $L(4,n)$ is a disjoint union of chains, where the corresponding chain tableaux are divided into four types:

A) $C_\mu$ with $\mu=(6k,4k,2k,0)$ for $k=0,1,\dots, \lfloor n/6 \rfloor$, as in Figure
\ref{L4nType1}. Partitions in these chains will be called of type $A_1,A_2,A_3,A_4$;

B) $C_\mu$ with $\mu=(6k+\ell,4k+\ell,2k,0)$, where $\ell \ge 2$ and $k=0,1,\dots,\lfloor (n-\ell)/6 \rfloor$, as in Figure \ref{L4nType2}.
Partitions in these chains will be called of type $B_1,B_3, B_2,B_4$;

C) $C_\mu$ with $\mu=(6k+r,4k,2k,0)$, where $r \ge 2$ and $k=0,1,\dots,\lfloor (n-r)/6 \rfloor$, as in Figure \ref{L4nType3}. Partitions in these chains will be called of type $C_{a2},C_{a3},C_{a4},C_{b2},C_{b_3},C_{b1}$;

D) $C_\mu$ with $\mu=(6k+r+\ell,4k+\ell,2k,0)$, where $\ell \ge 2$, $r \ge 2$ and $k=0,1,\dots,\lfloor (n-r-\ell)/6 \rfloor$, as in Figure \ref{L4nType4}.
Partitions in these chains will be called of type\\
 $D_{a3}, D_{a4}, D_{b2},D_{b3}, D_{c1},D_{c2}$.
\end{theo}

\begin{figure}[!ht]
\centering{
\includegraphics[height=1.6 in]{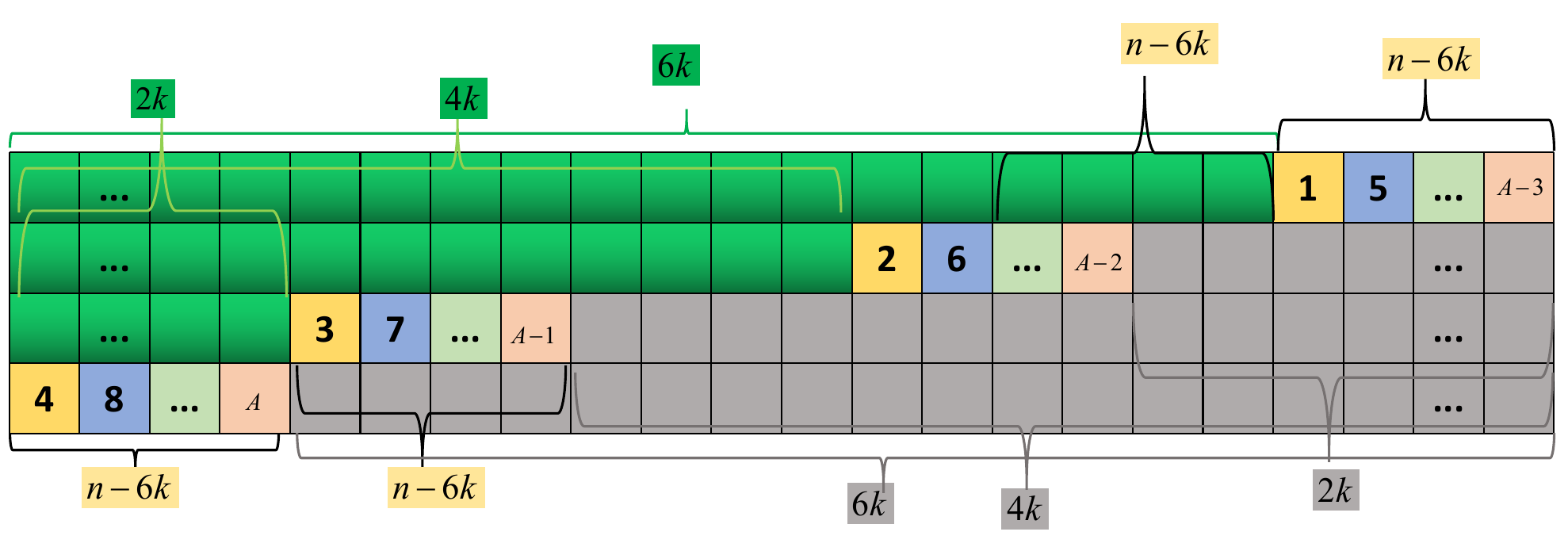}}
\caption{Chain tableaux of Type i) starting at $\mu=(6k,4k,2k,0)$,\ $k\in \mathbb{N}$, where $A=4n-24k$. }\label{L4nType1}
\end{figure}

\begin{figure}[!ht]
\centering{
\includegraphics[height=1.6 in]{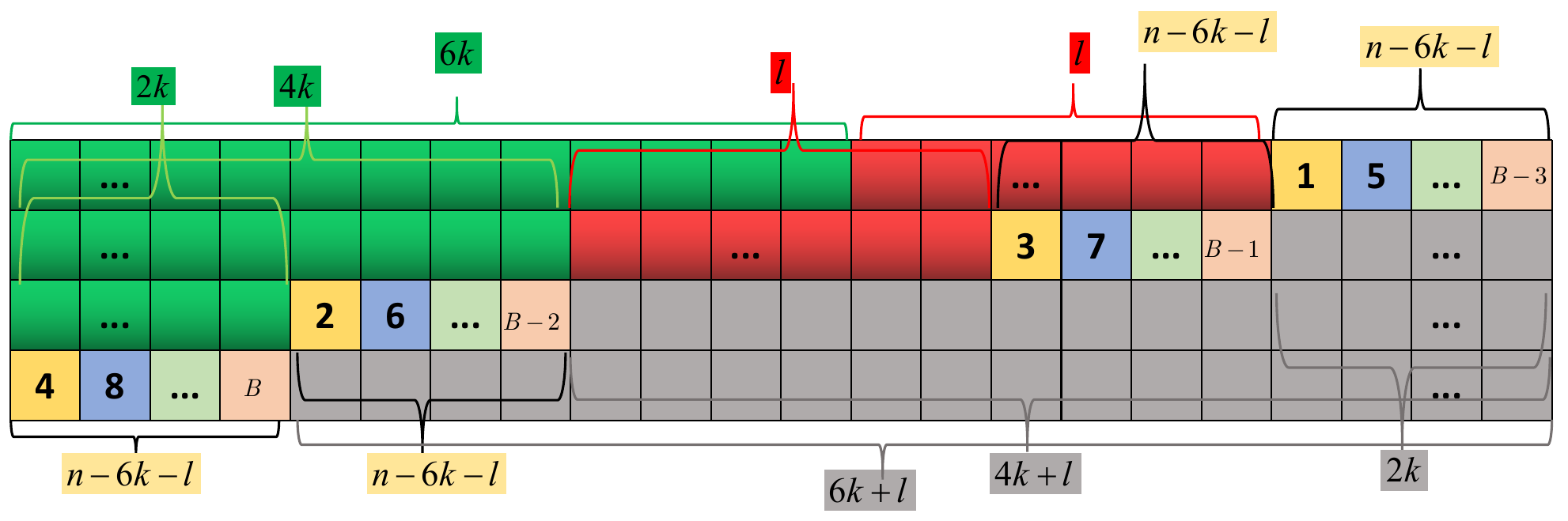}}
\caption{Chain tableaux of Type ii) starting at $\mu=(6k+\ell,4k+\ell,2k,0)$,\ $k\in \mathbb{N}$,$\ell\ge2$, where $B=4n-24k-4\ell$.\label{L4nType2}}
\end{figure}

\begin{figure}[!ht]
\centering{
\includegraphics[height=1.6 in]{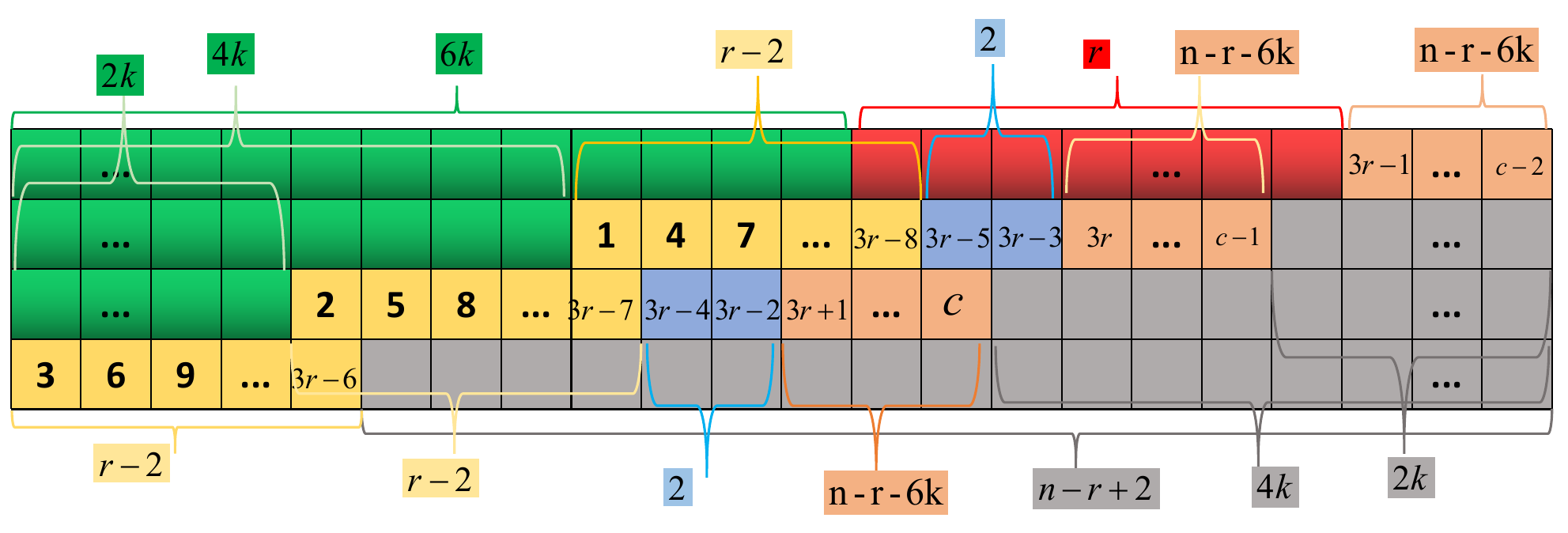}}
\caption{Chain tableaux of Type iii) starting at $\mu=(6k+r,4k,2k,0)$,\ $k\in \mathbb{N}$,$r\ge2$, where $C=3n-18k-2$.}\label{L4nType3}
\end{figure}

\begin{figure}[!ht]
\centering{
\includegraphics[height=1.6 in]{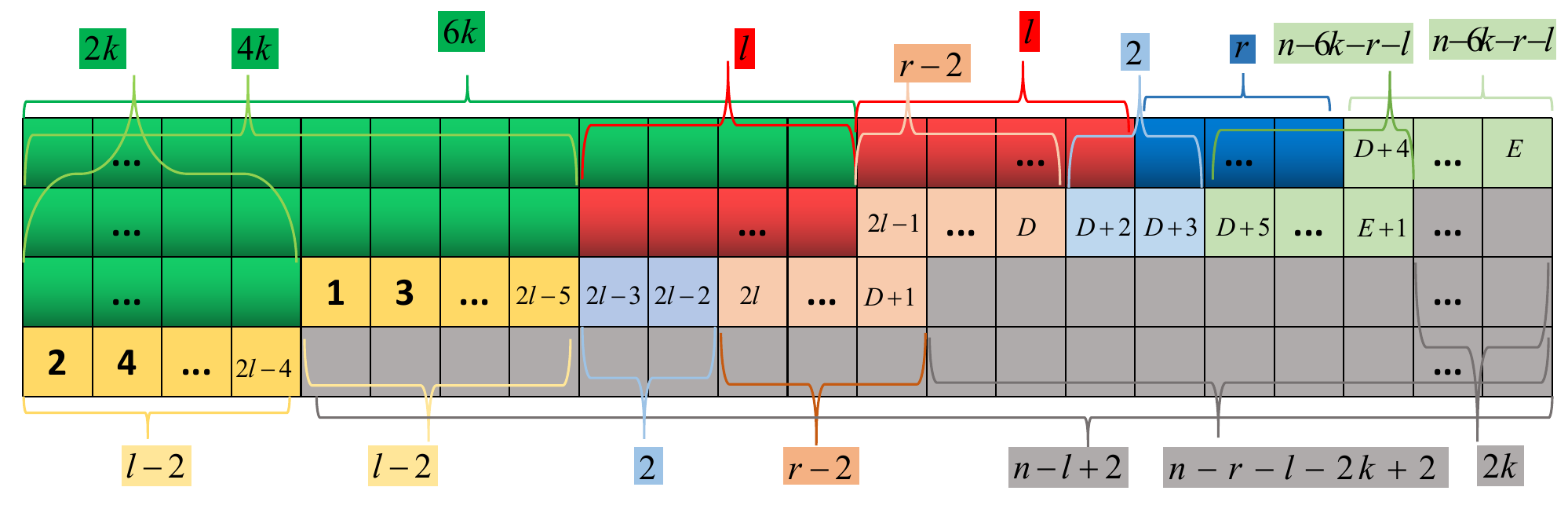}}
\caption{Chain tableaux of Type iv) starting at $\mu=(6k+r+\ell,4k+\ell,2k,0)$,\ $k\in \mathbb{N}$,$r\ge2$,$\ell\ge2$, where $D=2\ell+2r-7$ and $E=2n-12k-5$.}\label{L4nType4}
\end{figure}

The proof of Theorem \ref{theo:L4n-chain-dec} is similar to that of Theorem \ref{theo:varphi-chain}. We only sketch the idea. 

We divide partitions in $L(4,n)$ as in the following lemma.

\begin{lem}\label{lemTabEle}
Any element $\lambda=(\lambda_1,\lambda_2,\lambda_3,\lambda_4)$ in $L(4,n)$ can be expressed in one of the following forms:
\begin{enumerate}
  \item[$A_1)$] $(6k+c,4k+c,2k+c,c),k\geqslant0,n-6k\geqslant c\geqslant0$.
  \item[$A_2)$] $(6k+c+1,4k+c,2k+c,c),k\geqslant0,n-6k-1\geqslant c\geqslant0$.
  \item[$A_3)$] $(6k+c+1,4k+c+1,2k+c,c),k\geqslant0,n-6k-1\geqslant c\geqslant0$.
  \item[$A_4)$] $(6k+c+1,4k+c+1,2k+c+1,c),k\geqslant0,n-6k-1\geqslant c\geqslant0$.
  \item[$B_1)$] $(6k+\ell+c,4k+\ell+c,2k+c,c),k\geqslant0,\ell\geqslant2, n-6k-\ell \geqslant c \geqslant0$.
  \item[$B_3)$] $(6k+\ell+c+1,4k+\ell+c,2k+c,c),k\geqslant0,\ell\geqslant2,n-6k-\ell-1 \geqslant c \geqslant0$.
  \item[$B_2)$] $(6k+\ell+c+1,4k+\ell+c,2k+c+1,c),k\geqslant0,\ell\geqslant2,n-6k-\ell-1 \geqslant c \geqslant0$.
  \item[$B_4)$] $(6k+\ell+c+1,4k+\ell+c+1,2k+c+1,c),k\geqslant0,\ell\geqslant2,n-6k-\ell-1\geqslant c \geqslant0$.
  \item[$C_{a2})$] $(6k+r,4k+c,2k+c,c),k\geqslant0,r\geqslant2,r-2\geqslant c \geqslant0$.
  \item[$C_{a3})$] $(6k+r,4k+c+1,2k+c,c),k\geqslant0,r\geqslant2,r-2\geqslant c \geqslant0$.
  \item[$C_{a4})$] $(6k+r,4k+c+1,2k+c+1,c),k\geqslant0,r\geqslant2,r-3\geqslant c \geqslant0$.
  \item[$C_{b2})$] $(6k+r+c,4k+r-1+c,2k+r-1+c,r-2),k\geqslant0,r\geqslant2,n-r-6k\geqslant c \geqslant0$.
  \item[$C_{b3})$] $(6k+r+c,4k+r+c,2k+r-1+c,r-2),k\geqslant0,r\geqslant2,n-r-6k\geqslant c \geqslant0$.
  \item[$C_{b1})$] $(6k+r+c,4k+r+c,2k+r+c,r-2),k\geqslant0,r\geqslant2,n-r-6k\geqslant c \geqslant0$.
  \item[$D_{a3})$] $(6k+\ell+r,4k+\ell,2k+c,c),k\geqslant0,\ell\geqslant2,r\geqslant2,\ell-2 \geqslant c\geqslant0$.
  \item[$D_{a4})$] $(6k+\ell+r,4k+\ell,2k+c+1,c),k\geqslant0,\ell\geqslant2,r\geqslant2,\ell-2 \geqslant c\geqslant0$.
  \item[$D_{b2})$] $(6k+\ell+r,4k+\ell+c,2k+\ell+c,\ell-2),k\geqslant0,\ell\geqslant2,r\geqslant2,r-2 \geqslant c\geqslant0$.
  \item[$D_{b3})$] $(6k+\ell+r,4k+\ell+c+1,2k+\ell+c,\ell-2),k\geqslant0,\ell\geqslant2,r\geqslant2,r-2 \geqslant c\geqslant0$.
  \item[$D_{c1})$] $(6k+\ell+r+c,4k+\ell+r+c,2k+\ell+r-2,\ell-2),k\geqslant0,\ell\geqslant2,r\geqslant2,n-6k-r-\ell \geqslant c\geqslant0$.
  \item[$D_{c2})$] $(6k+\ell+r+c+1,4k+\ell+r+c,2k+\ell+r-2,\ell-2),k\geqslant0,\ell\geqslant2,r\geqslant2,n-6k-r-\ell-1 \geqslant c\geqslant0$.
\end{enumerate}
\end{lem}

\begin{proof}
Clearly partitions in each type are different from each other. To show that partitions from different types can not equal, we compute in the following table, where
for $\lambda=(\lambda_1,\lambda_2,\lambda_3,\lambda_4)$, we define $\alpha=\lambda_1-\lambda_2,\beta =\lambda_2-\lambda_3,\gamma=\lambda_3-\lambda_4$.
The proof is completed by data in Figures \ref{fig:abc-types}, \ref{L4nProof}, \ref{fig:type-abc}, and \ref{fig:type-abc-lambda}.

\begin{figure}[!ht]
\[\begin{array}{|l|l|l|l|l|l|l|}\hline
 \lambda  & \alpha-\beta        & \alpha-\gamma          & \beta-\gamma             &  \alpha               & \beta            &\gamma             \\\hline
A_1) & 0                         & 0                      & 0                        &  even                 & even              & even           \\\hline
A_2) & 1                         & 1                      & 0                        &  odd                  & even              & even          \\\hline
A_3) & -1                        & 0                      & 1                        &  even                 & odd               & even          \\\hline
A_4) & 0                         & -1                     & -1                       &  even                 & even              & odd          \\\hline
B_1) & -\ell \leqslant -2        & 0                      & \ell\geqslant 2          &  even                 &                   & even          \\\hline
B_3) & 1-\ell\leqslant -1        & 1                      & \ell\geqslant 2          &  odd                  &                   & even          \\\hline
B_2) & 2-\ell\leqslant 0         & 0                      & \ell-2\geqslant 0        &  odd                  &                   & odd          \\\hline
B_4) & -\ell \leqslant -2        & -1                     & \ell-1\geqslant 1        &  even                 &                   & odd          \\\hline
C_{a2}) & r-c\geqslant2          & r-c\geqslant 2         & 0                        &                       & even              & even          \\\hline
C_{a3}) & r-c-2\geqslant 0       & r-c-1\geqslant1        & 1                        &                       & odd               & even          \\\hline
C_{a4}) & r-c-1\geqslant2        & r-c-2\geqslant1        & -1                       &                       & even              & odd          \\\hline
C_{b2}) & 1                      & -c  \leqslant-1        & -1-c\leqslant-1          &  odd                  & even              &           \\\hline
C_{b3}) & -1                     & -1-c\leqslant-1        & -c\leqslant0             &  even                 & odd               &           \\\hline
C_{b1}) & 0                      & -2-c\leqslant-2        & -2-c\leqslant-2          &  even                 & even              &           \\\hline
D_{a3}) & r-\ell+c               & r\geqslant2            & \ell-c\geqslant2        &                       &                   & even          \\\hline
D_{a4}) & r-\ell+c+1             & r-1\geqslant1          & \ell-c-2\geqslant0      &                       &                   & odd          \\\hline
D_{b2}) & r-c\geqslant2          & r-2c-2                 & -2-c\leqslant-2         &                       &  even             &           \\\hline
D_{b3}) & r-c-2\geqslant0        & r-2c-3                 & -1-c\leqslant-1         &                       &  odd              &           \\\hline
D_{c1}) & -2-c\leqslant-2        & -r\leqslant-2          & c+2-r                   &  even                 &                   &           \\\hline
D_{c2}) & -1-c\leqslant-1        & 1-r\leqslant-1         & c+2-r                   &  odd                  &                   &           \\\hline
\end{array}\]
\caption{The $\alpha, \beta,\gamma$ values for different types.}\label{fig:abc-types}
\end{figure}

\begin{figure}[!ht]
\centering{
\includegraphics[height=5 in]{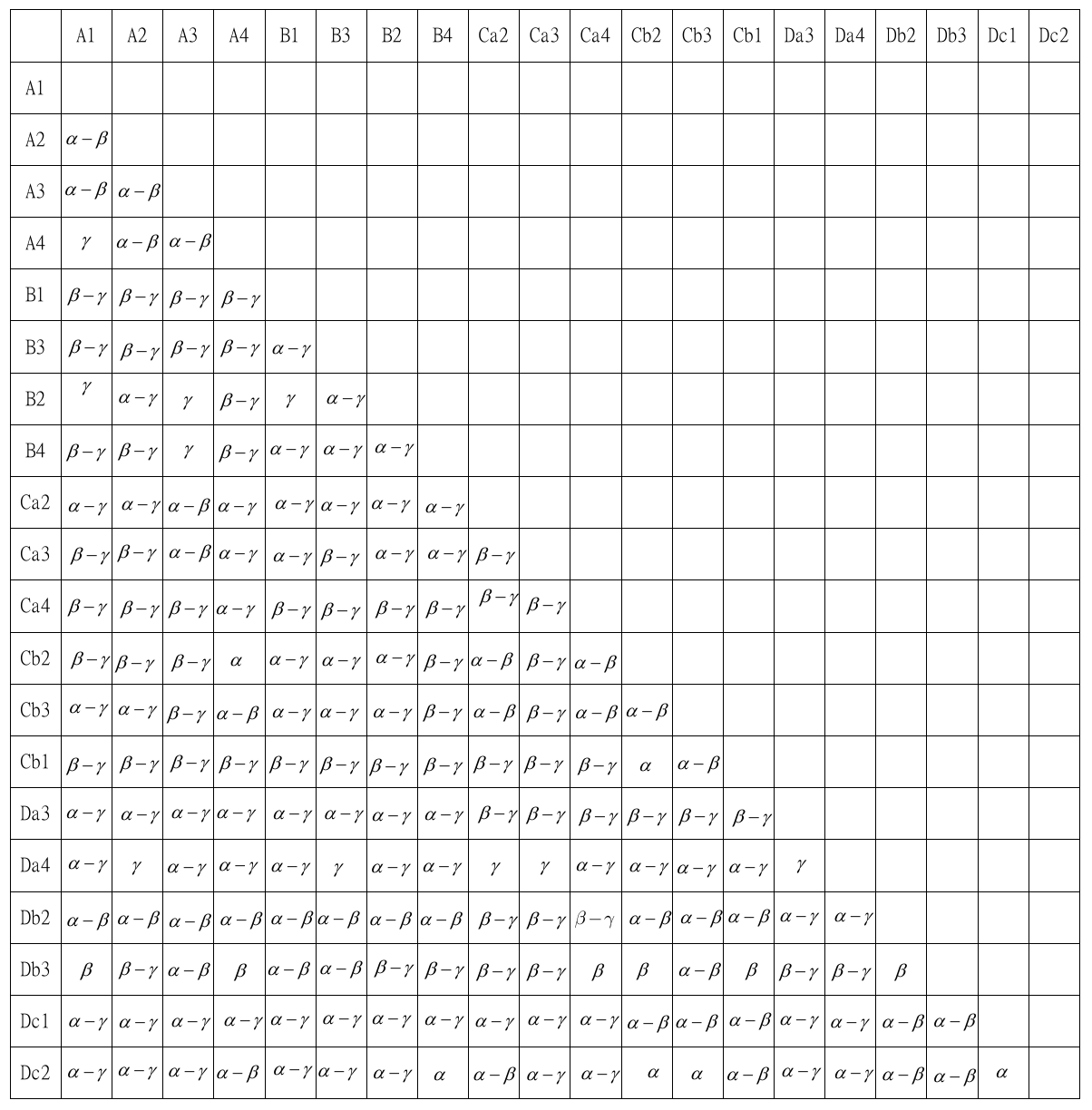}}
\caption{Proof of the distinction of partitions from different types. }\label{L4nProof}
\end{figure}

\begin{figure}[!ht]

$$\begin{array}{|l|l|l|l|l|l|l|}\hline
 \alpha-\beta   & \alpha-\gamma  & \beta-\gamma   &  \alpha  & \beta  &\gamma &\lambda    \\\hline
 0              & 0              & 0              &  even    & even   & even  &A_1        \\\hline
 1              & 1              & 0              &  odd     & even   & even  &A_2         \\\hline
 -1             & 0              & 1              &  even    & odd    & even  &A_3         \\\hline
 0              & -1             & -1             &  even    & even   & odd   &A_4            \\\hline
 \leqslant -2   & 0              & \geqslant 2    &  even    &        & even  &B_1  \\\hline
 \leqslant -1   & 1              & \geqslant 2    &  odd     &        & even  &B_3   \\\hline
 \leqslant 0    & 0              & \geqslant 0    &  odd     &        & odd   &B_2\\\hline
 \leqslant -2   & -1             & \geqslant 1    &  even    &        & odd   &B_4   \\\hline
 \geqslant2     & \geqslant 2    & 0              &          & even   & even  &C_{a2}  \\\hline
 \geqslant 0    & \geqslant1     & 1              &          & odd    & even  &C_{a3} \\\hline
 \geqslant2     & \geqslant1     & -1             &          & even   & odd   &C_{a4} \\\hline
 1              & \leqslant0     & \leqslant-1    &  odd     & even   &       &C_{b2} \\\hline
  -1            & \leqslant-1    & \leqslant0     &  even    & odd    &       &C_{b3} \\\hline
 0              & \leqslant-2    & \leqslant-2    &  even    & even   &       &C_{b1}  \\\hline
                & \geqslant2     & \geqslant2     &          &        &even   &D_{a3}     \\\hline
                & \geqslant1     & \geqslant0     &          &        & odd   &D_{a4}    \\\hline
\geqslant2      &                & \leqslant-2    &          &  even  &       &D_{b2}  \\\hline
 \geqslant0     &                & \leqslant-1    &          &  odd   &       &D_{b3}  \\\hline
 \leqslant-2    & \leqslant-2    &                &  even    &        &       &D_{c1}  \\\hline
 \leqslant-1    & \leqslant-1    &                &  odd     &        &       &D_{c2}  \\\hline
\end{array}
$$
\caption{Determine the types from the $\alpha,\beta,\gamma$ values.}\label{fig:type-abc}
\end{figure}

\begin{figure}[!ht]
$$\begin{array}{|l|l|l|l|l|}\hline
\lambda      &k                   &\ell                              &r                                 &c  \\\hline
A_1          &\frac{\alpha}{2}    &\times                            &\times                            &\lambda_4  \\\hline
A_2          &\frac{\beta}{2}     &\times                            &\times                            &\lambda_4  \\\hline
A_3          &\frac{\alpha}{2}    &\times                            &\times                            &\lambda_4        \\\hline
A_4          &\frac{\alpha}{2}    &\times                            &\times                            &\lambda_4      \\\hline
B_1          &\frac{\alpha}{2}    &-(\alpha-\beta)                   &\times                            &\lambda_4        \\\hline
B_3          &\frac{\alpha-1}{2}  &-(\alpha-\beta)+1                 &\times                            &\lambda_4\\\hline
B_2          &\frac{\alpha-1}{2}  &-(\alpha-\beta)+2                 &\times                            &\lambda_4\\\hline
B_4          &\frac{\alpha}{2}    &-(\alpha-\beta)                   &\times                            &\lambda_4\\\hline
C_{a2}       &\frac{\beta}{2}     &\times                            &(\alpha-\beta)+\lambda_4          &\lambda_4\\\hline
C_{a3}       &\frac{\beta-1}{2}   &\times                            &(\alpha-\beta)+\lambda_4+2        &\lambda_4\\\hline
C_{a4}       &\frac{\beta}{2}     &\times                            &\lambda_4+2                       &\lambda_4\\\hline
C_{b2}       &\frac{\beta}{2}     &\times                            &\lambda_4+2                       &-(\beta-\gamma)-1\\\hline
C_{b3}       &\frac{\alpha}{2}    &\times                            &\lambda_4+2                       &-(\alpha-\gamma)-1\\\hline
C_{b1}       &\frac{\alpha}{2}    &\times                            &\lambda_4+2                       &-(\alpha-\gamma)-2\\\hline
D_{a3}       &\frac{\gamma}{2}    &(\beta-\gamma)+\lambda_4           &(\alpha-\gamma)                   &\lambda_4\\\hline
D_{a4}       &\frac{\gamma-1}{2}  &(\beta-\gamma)+\lambda_4+2         &(\alpha-\gamma)+1                 &\lambda_4\\\hline
D_{b2}       &\frac{\beta}{2}     &\lambda_4+2                      &(\alpha-\beta)-(\beta-\gamma)-2    &-(\beta-\gamma)-2\\\hline
D_{b3}       &\frac{\beta-1}{2}   &\lambda_4+2                      &(\alpha-\beta)-(\beta-\gamma)+1    &-(\beta-\gamma)-1\\\hline
D_{c1}       &\frac{\alpha}{2}    &\lambda_4+2                       &(\alpha-\gamma)                   &-(\alpha-\beta)-2\\\hline
D_{c2}       &\frac{\alpha-1}{2}  &\lambda_4+2                       &(\alpha-\gamma)+1                 &-(\alpha-\beta)-1\\\hline
\end{array}
$$
\caption{From the types and $\alpha,\beta,\gamma$ values to the partitions.}\label{fig:type-abc-lambda}
\end{figure}
\end{proof}

\section{Recursive Sperner chain decompositions of $L(m,n)$ \label{sec:RecursiveSperner}}
A basic idea is that by using the dual operation,  we only need the upper half part of $L(m,n)$ to construct the a Sperner chain decomposition.
\begin{defn}
$L^{U}(m,n)=:\{\lambda\mid\lambda\in L(m,n), \left| \lambda \right|\le d_{m,n}, \ d_{m,n}=\left \lfloor \frac{mn+1}{2} \right \rfloor\}$.
\end{defn}

\begin{defn}
A U-decomposition of $L^{U}(m,n)$ is a chain decomposition $C_1,C_2,\cdots,C_k$, where $C_i= \lambda_{i,1} \le \lambda_{i,2} \le \cdots \le \lambda_{i,t_{i}} ,\ i=1,2,\cdots,k$ satisfy the following conditions:

(1) $L^{U}(m,n)=C_1\cup C_2\cup \cdots \cup C_k$;

(2) $C_1,C_2,\cdots,C_k$ are  disjoint;

(3) $\rank(\lambda_{i,t_i})=d_{m,n}$.
\end{defn}

\begin{exam}
 The U decompositions of $L^{U}(2,5)$ and $L^{U}(2,8)$ are given in Figure \ref{L25L28U}.

\begin{figure}[!ht]
\centering{
\includegraphics[height=2.5 in]{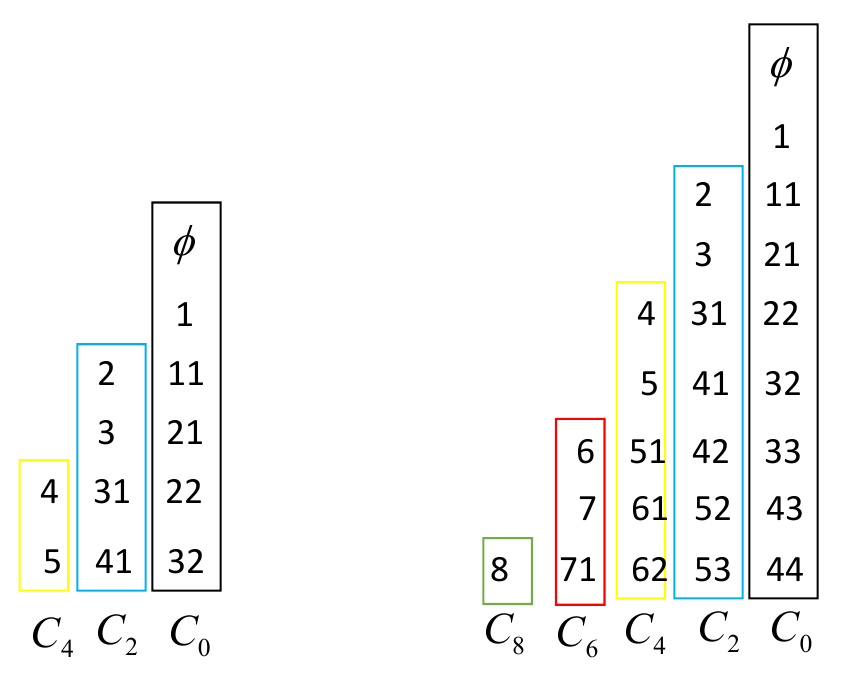}}
\caption{The U decomposition of $L^{U}(2,5)$ and $L^{U}(2,8)$.}\label{L25L28U}
\end{figure}
\end{exam}

Clearly, the dual of a U-decomposition of $L^{U}(m,n)$ is also a chain decomposition, but of the lower half of $L(m,n)$.
We can knead them together to obtain a Sperner chain decomposition of $L(m,n)$.

\begin{lem}\label{SpernerChain}
If $L^{U}(m,n)$ has a U-decomposition $\{C_1,C_2,\cdots,C_k\}$, where $C_i$ is $\lambda_{i,1} \le \lambda_{i,2} \le \cdots  \le \lambda_{i,t_{i}}(i=1,2,\cdots,k)$,
then $L(m,n)$ has a Sperner chain decomposition,
and $\{\lambda_{i,1},i=1,2,\cdots,k\}$  is the starting set for the chain decomposition,
$\{\lambda^*_{i,1},i=1,2,\cdots,k\}$ is the end point set of the chain decomposition.
\end{lem}

\begin{proof}
we discuss the parity of $mn$ as follows.
\begin{itemize}
  \item [(1).] When $mn$ is even, since the proposition $L(m,n)$ is rank-symmetric, we get $\{\lambda_{i,t_i},i=1,2,\cdots,k\}=\{\lambda^*_{i,t_i},i=1,2,\cdots,k\}$.
             So we knead the chain $C=\{C_1,C_2,\cdots,C_k\}$ and the chain $C^{\star} = \{C^{\star}_1,C^{\star}_2,\cdots,C^{\star}_k \}$ together to form a new chain denoted as $\tilde{C}=\{\tilde{C}_1,\tilde{C}_2,...,\tilde{C}_k\}$.
  \item [(2).] When $mn$ is odd, since the proposition $L(m,n)$ is rank-symmetric, we get $\{\lambda_{i,t_i},i=1,2,\cdots,k\}=\{\lambda^*_{i,t_i-1},i=1,2,\cdots,k\}$.
             So we knead the chain $C=\{C_1,C_2,\cdots,C_k\}$ and the chain $C^{\star} = \{C^{\star}_1,C^{\star}_2,\cdots,C^{\star}_k \}$ together to form a new chain denoted as $\tilde{C}=\{\tilde{C}_1,\tilde{C}_2,\cdots,\tilde{C}_k\}$.
\end{itemize}
In a word, we get the Sperner chain decomposition $\tilde{C}=\{\tilde{C}_1,\tilde{C}_2,\cdots,\tilde{C}_k\}$ of $L(m,n)$.
\end{proof}

The above Lemma \ref{SpernerChain} is best explained by an example.
\begin{exam}
The U-decomposition of $L^{U}(3,3)$ is given by $C_1, C_2, C_3$. Then $C^*_1,C^*_2,C^*_3$ are constructed.
We get the Sperner chain decomposition $\tilde{C}=\{\tilde{C}_1,\tilde{C}_3,\tilde{C}_3\}$ of $L(3,3)$. See Figure \ref{fig:U-dec-L33}.

\begin{figure}[!ht]
\centering{
\includegraphics[height=3.5 in]{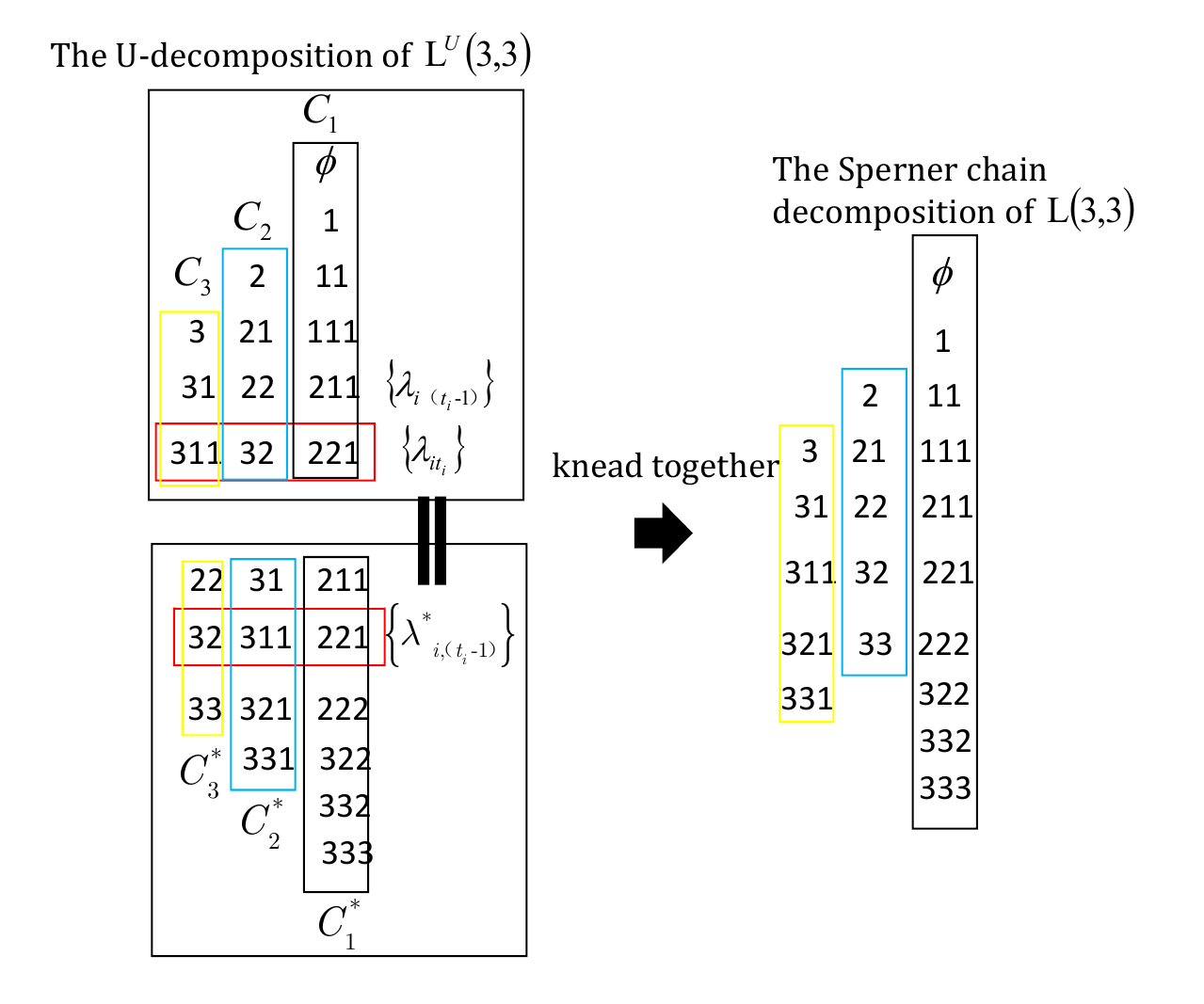}}
\caption{The knead process for $L(3,3)$.   }\label{fig:U-dec-L33}
\end{figure}

\end{exam}


\subsection{The recursive construction}
To decompose $L^U(m,n)$, we use the natural recursion $L(m,n)= L(m,n-1) \biguplus (n\oplus L(m-1, n))$, where
$n\oplus (\lambda_1,\dots, \lambda_{m-1})=(n,\lambda_1,\dots, \lambda_{m-1})$ and $n\oplus L(m-1,n)=\{n\oplus \lambda: \lambda \in L(m- 1,n)\}$.

\noindent
\textbf{Algorithm RecUD}

\noindent
\textbf{Input}: The U-decompositions of $L^U(m,n-1)$ and $L^U(m-1,n)$.\\
\noindent
\textbf{Output}: The U-decomposition of $L^U(m,n)$ if possible.

\begin{enumerate}
\item Construct the Sperner Chain decomposition of $L(m,n-1)$ from the U-decomposition of $L^U(m,n-1)$.
Chains end at a  rank less than $d_{m,n}$ is called bad chains and need further operation.
Other chains are cut at rank $d_{m,n}$ and kept as good chains.
Let $E=\{\alpha^1,\dots, \alpha^e\}$ be the set of end partitions of bad chains.

\item Map $\lambda$ to $n\bigoplus \lambda$ for all partitions in the U-decomposition of $L^U(m-1,n)$.
Cut all chains at rank $d_{m,n}$. These are candidate chains
for kneading with bad chains.
Let $S=\{\beta^1,\dots, \beta^s\}$ be the set of all starting partitions of these chains.

\item For each $\alpha^i$, add $1$ to the first entry and check if it equals $\beta^{j_i}$.
If true for all $i$ then knead the bad chains
with the corresponding candidate chains.

\item Good chains, kneaded candidate chains, and the remaining candidate chains form the U-decomposition of $L^U(m,n)$.
\end{enumerate}

\begin{exam}
The U-decomposition of $L^{U}(3,5)$ is constructed in Figures \ref{Step1},\ref{Step2},\ref{Step3},\ref{Step4}.
\begin{figure}[!ht]
\centering{
\includegraphics[height=3.0 in]{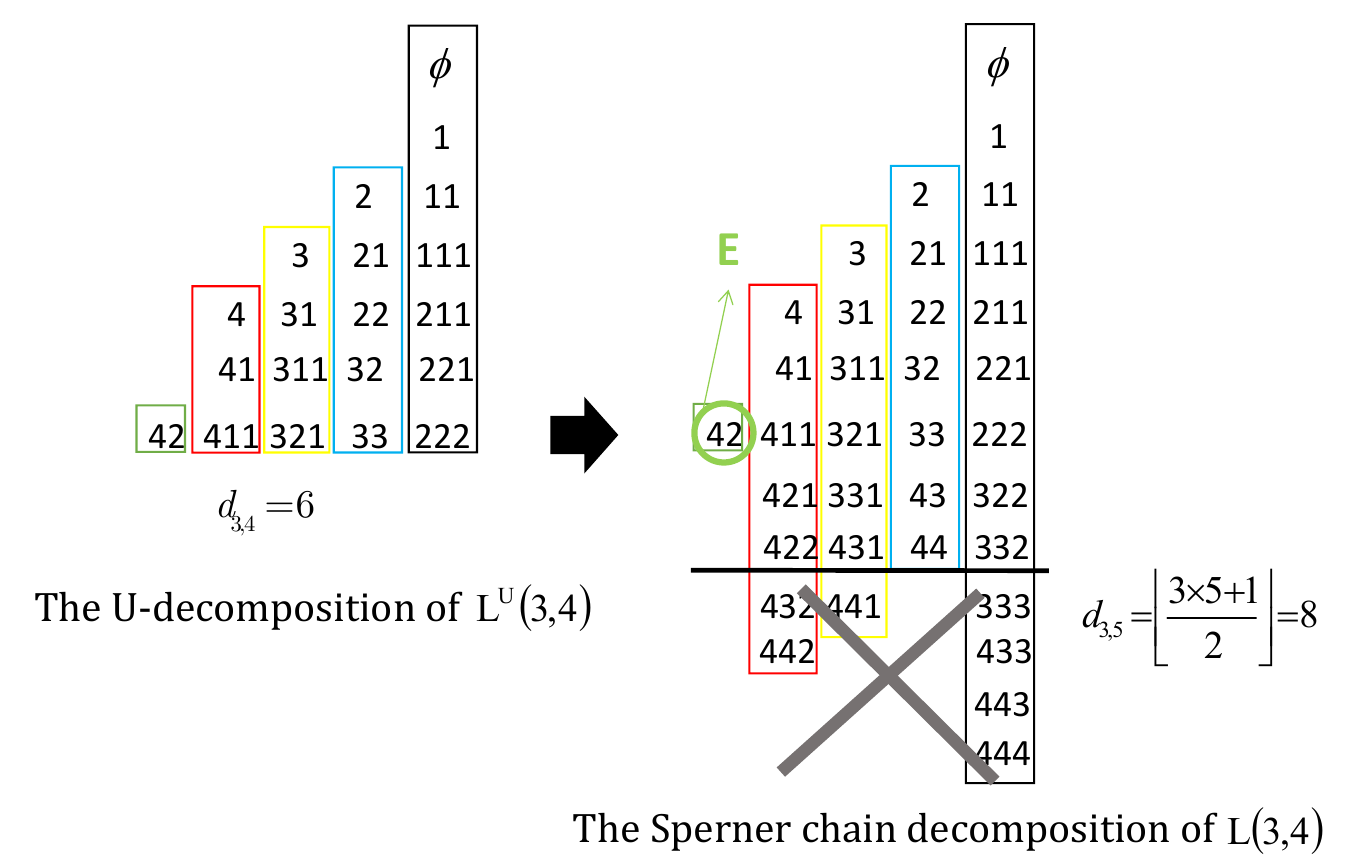}}
\caption{Step 1}\label{Step1}
\end{figure}

\begin{figure}[!ht]
\centering{
\includegraphics[height=2.2 in]{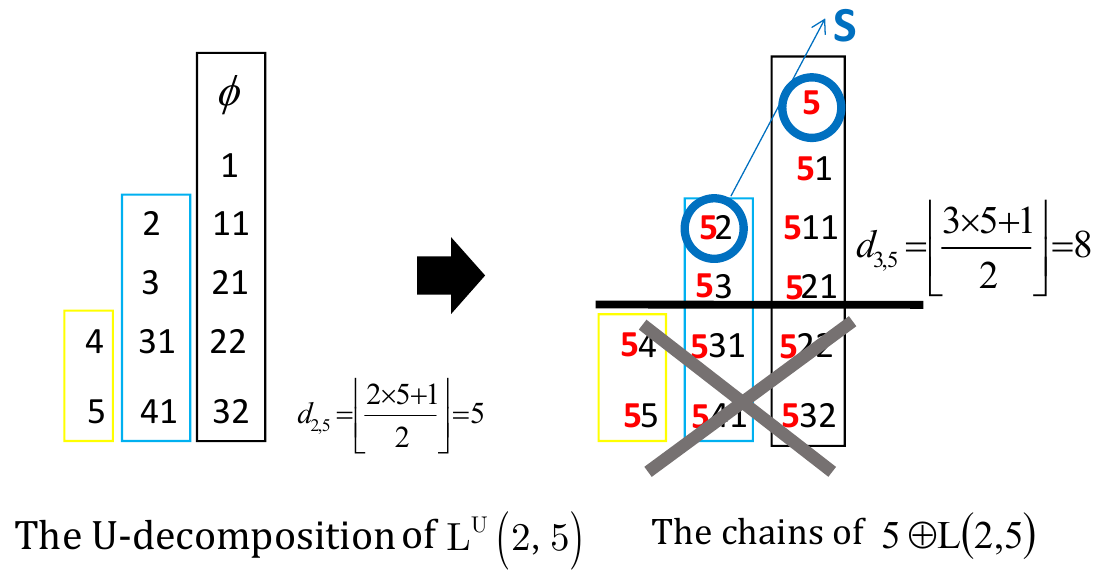}}
\caption{Step 2}\label{Step2}
\end{figure}

\begin{figure}[!ht]
\centering{
\includegraphics[height=2.2 in]{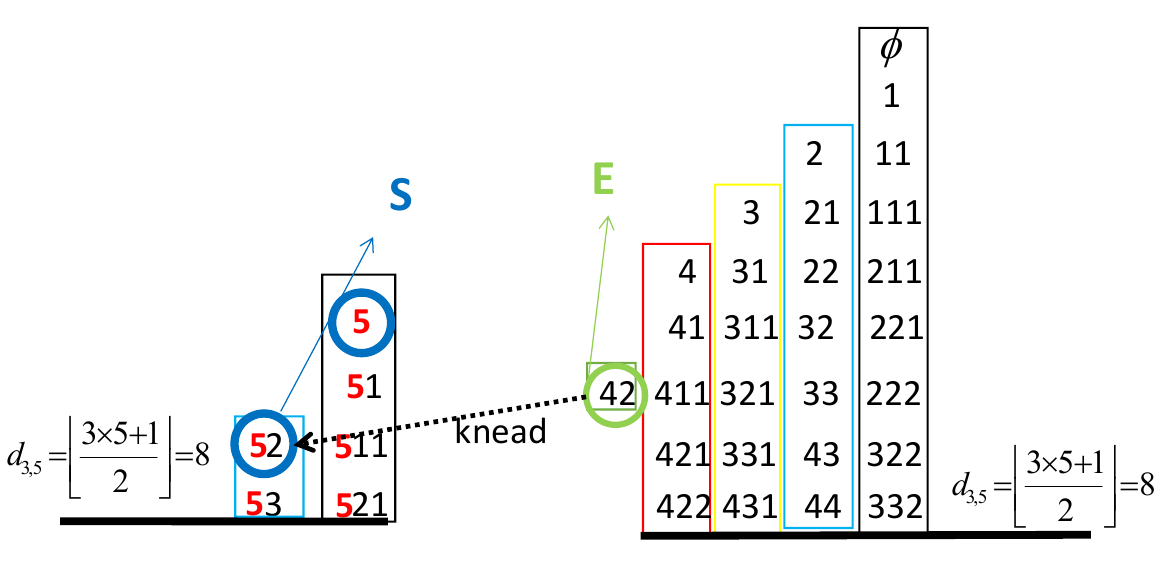}}
\caption{Step 3}\label{Step3}
\end{figure}

\begin{figure}[!ht]
\centering{
\includegraphics[height=2.3 in]{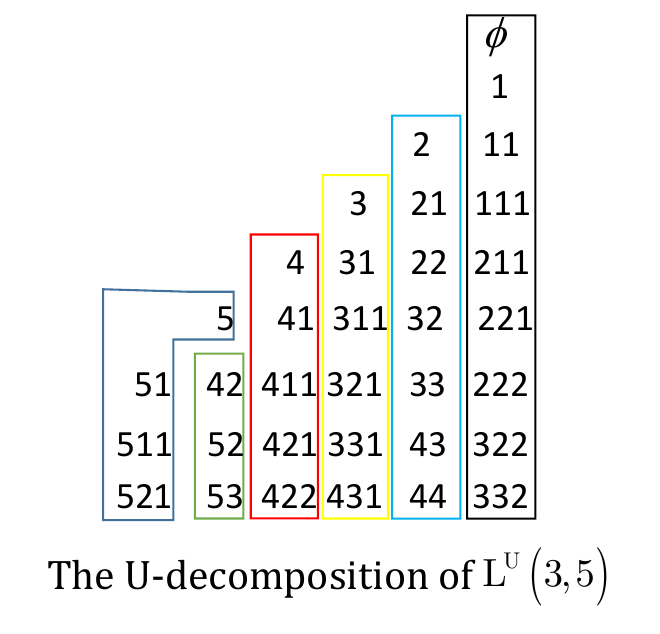}}
\caption{Step 4}\label{Step4}
\end{figure}
\end{exam}

From the algorithm, we see that the starting partitions play an important role. If we only want to
show the existence of a Sperner chain decomposition, the algorithm maybe much simpler.
Indeed, we can obtain the set $S_{m,n}$ of starting partitions .

\noindent
\textbf{Algorithm RecSmn}

\noindent
\textbf{Input}: The sets $S_{m,n-1}$ and $S_{m-1,n}$ of starting partitions for certain Sperner
chain decompositions of $L(m,n-1)$ and $L(m-1,n)$, respectively.

\noindent
\textbf{Output}: The set $S_{m,n}$ of starting partitions for a possible Sperner chain decomposition of $L(m,n)$.

\begin{enumerate}
\item Let $E_{m,n-1}=S_{m,n-1}^*$ be the set of end partitions of a Sperner chain decomposition of $L(m,n-1)$.
Select all partitions of rank less than $d_{m,n}$ to form $E=\{\alpha^1,\dots, \alpha^e\}$.

\item Let $S=n\bigoplus S_{m-1,n}=\{\beta_1,\beta_2,\dots\}$.

\item Add $1$ to the first entry for each $\alpha^i$ to obtain $E'$.
If $E'\subseteq S$ then we can construct a Sperner chain decomposition of $L(m,n)$
with $S_{m,n}= S_{m,n-1} \bigcup S \setminus E'$. Otherwise, the method failed.
\end{enumerate}

The algorithm succeeds for $m\le 4$.

Clearly, for $m=1$ $L(1,n)$ is a chain and we have $S_{1,n}=\{(0)\}$.

For $m=2$, we have

\begin{prop}
  There is a Sperner chain decomposition for $L(2,n)$ with starting partitions
  $S_{2,n}=\{(2k,0): 0\le k \le n/2 \}$.
\end{prop}
\begin{proof}
The proposition clearly holds for $n=1$, in which case
$L(2,1)$ itself is a chain.

Assume the proposition holds for $n$ and we want to show that it holds for $n+1$.
We use Algorithm RecSmn and discuss by the parity of $n$:
i) If $n$ is odd, then no elements in $E_{2,n}$ has rank less than $d_{2,n+1}$, so that $E=\emptyset$ and
$S=(n+1)\oplus(0)$. It follows that $S_{2,n+1}=S_{2,n} \cup \{(n+1,0)\}$ as desired;
ii) If $n$ is even, then only one element $(n,0)^*=(n,0)$ has rank less than $d_{2,n+1}=n+1$, so that $E=\{(n,0)\}$ and
$S=(n+1)\oplus(0)$. Clearly we have the match $(n,0)\lessdot (n+1,0)$. It follows that $S_{2,n+1}=S_{2,n}$ as desired.
\end{proof}

For $m=3$, we have
\begin{prop}
  There is a Sperner chain decomposition for $L(3,n)$ with starting partitions $S_{3,n}$ as in \eqref{e-E3n}.
\end{prop}
\begin{proof}
To apply Algorithm RecSmn, it is better to have the following facts as guide for our proof.
\begin{enumerate}
  \item $S_{3,4t}=S_{3,4t-1}\cup 4t\oplus \{(4t,0),(4t,2),(4t,4),\cdots,(4t,2t)\}$.
  \item $S_{3,4t+1}=S_{3,4t}\cup(4t+1)\oplus \{(4t+1,0),(4t+1,2),(4t+1,4),\cdots,(4t+1,2t-2)\}$.
  \item $S_{3,4t+2}=S_{3,4t+1}\cup(4t+2)\oplus \{(4t+2,0),(4t+2,2),(4t+2,4),\cdots,(4t+2,2t)\}$.
  \item $S_{3,4t+3}=S_{3,4t+2}\cup(4t+3)\oplus \{(4t+3,0),(4t+3,2),(4t+3,4),\cdots,(4t+3,2t)\}$.
\end{enumerate}
Case 1, for $(m,n)=(3,4t)$. We only need to consider starting partitions whose ranks $r$ satisfy $3\cdot (4t-1)-r<d_{3,4t}=6t$,
which is equivalent to $r>6t-3$. By definition, we need all $(4k+\ell, 2k,0)$ with $\ell \neq 1$
 satisfying $6k+\ell>6t-3$ and $4k+\ell \le 4t-1$, which simplifies to
 $t - \frac{\ell+3}{6}   <   k \le t-\frac{\ell+1}{4}$. Since no such integer $k$ exists, $E$ is empty, and
 we have $S_{3,4t}=S_{3,4t-1} \cup ((4t+1)\oplus S_{2,4t})$ as desired.

 Case 2, for $(m,n)=(3,4t+1)$.
We only need to consider starting partitions whose ranks $r$ satisfy $3\cdot 4t-r<d_{3,4t+1}=6t+2$,
which is equivalent to $r>6t-2$. By definition, we need all $(4k+\ell, 2k,0)$ with $\ell \neq 1$
 satisfying $6k+\ell>6t-2$ and $4k+\ell \le 4t$, which simplifies to
 $t - \frac{\ell+2}{6}   <   k \le t-\frac{\ell}{4}$. This can happen only when $k=t$ and $\ell=0$. So $E=\{(4t,2t,0)^*\}=\{(4t,2t,0)\}$.
 Then $E'=\{(4t+1,2t,0)\} \subseteq(4t+1)\oplus S_{2,4t+1}$.
Hence we have $S_{3,4t+1}=S_{3,4t} \cup ((4t+1)\oplus S_{2,4t+1}) \setminus \{(4t+1,2t,0)\}$ as desired.

The other cases are similar.
\end{proof}
%

\begin{theo}
If we use the Sperner chain decomposition in Theorem \ref{theo-varphi}, then Algorithm RecUD
gives rise a Sperner chain decomposition of $L(4,n)$ with starting partitions
$$ S_{4,n} = \{ (6k+s+\ell, 4k+\ell, 2k,0): s,\ell,k\ge 0,\  s\neq 1,\ \ell\neq 1, \text{ and } 6k+s+\ell \leq n\}.$$
\end{theo}
\begin{proof}

Case 1. We explain how to obtain $S_{4,6t+1}$ from $S_{4,6t}$.

We first need to find all starting partitions in $S_{4,6t}$ whose rank $r$ satisfy $4\cdot 6t -r < d_{4,6t+1}= 12t+2$,
which is $r>12t-2$. Let $(6k+s+\ell, 4k+\ell, 2k,0)$ be a such partition. Then we need the conditions:
$12k+s+2\ell >12 t-2$, $6k+s+\ell\le 6t$, $s\neq 1$, $\ell\neq 1$. The first two inequalities are equivalent to
$ t- \frac{1+s+2\ell }{12}    \le k  \le t-\frac{s+\ell}{6}$, which implies $s\le 1$ and hence $s=0$.
Now it is easy to see that $\ell $ has to be a multiple of $6$ and $k=t-\ell/6$.
By listing all such partitions and computing their dual, we obtain
\begin{align*}
 E=\{(6t,6t,0,0), (6t,6t-2,2,0), (6k,6k-4,4,0), \cdots, (6k,4k,2k,0)\}.
\end{align*}
Then one easily checked that
$$E'=\{(6t+1,6t,0,0), (6t+1,6t-2,2,0), (6k+1,6k-4,4,0), \cdots, (6k+1,4k,2k,0)\}$$ and $E'\subseteq (6t+1)\bigoplus S_{3,6t+1}$.
Hence we have $S_{4,6t+1}=S_{4,6t} \cup ((6t+1)\bigoplus S_{3,6t+1}) \setminus E'$ as desired.

Case 2. We explain how to obtain $S_{4,6t+2}$ from $S_{4,6t+1}$.

We first need to find all starting partitions in $S_{4,6t+1}$ whose rank $r$ satisfy $4\cdot (6t+1) -r < d_{4,6t+2}= 12t+4$,
which is $r>12t$. Let $(6k+s+\ell, 4k+\ell, 2k,0)$ be a such partition. Then we need the conditions:
$12k+s+2\ell >12 t$, $6k+s+\ell\le 6t+1$, $s\neq 1$, $\ell\neq 1$. The first two inequalities are equivalent to
$ t- \frac{s+2\ell }{12}    \le k  \le t-\frac{s+\ell-1}{6}$, which implies $s< 2$ and hence $s=0$.
Now it is easy to see that $\ell $ has to be a multiple of $6$ and $k=t-\frac{\ell-1}{6}$.
By listing all such partitions and computing their dual, we obtain
$$
E=\{(6t+1,6t+1,0,0), (6t+1,6t-1,2,0), (6k+1,6k-3,4,0), \cdots, (6k+1,4k+3,2k-2,0)\}.
$$
Then one easily checked that
$$E'=\{(6t+2,6t+1,0,0), (6t+2,6t-1,2,0), (6k+2,6k-3,4,0), \cdots, (6k+2,4k+3,2k-2,0)\}$$ and $E'\subseteq (6t+2)\bigoplus S_{3,6t+2}$.
Hence we have $S_{4,6t+2}=S_{4,6t+1} \cup ((6t+2)\bigoplus S_{3,6t+2}) \setminus E'$ as desired.

Case 3. We explain how to obtain $S_{4,6t+3}$ from $S_{4,6t+2}$.

We first need to find all starting partitions in $S_{4,6t+2}$ whose rank $r$ satisfy $4\cdot (6t+2) -r < d_{4,6t+3}= 12t+6$,
which is $r>12t+2$. Let $(6k+s+\ell, 4k+\ell, 2k,0)$ be a such partition. Then we need the conditions:
$12k+s+2\ell >12 t+2$, $6k+s+\ell\le 6t+2$, $s\neq 1$, $\ell\neq 1$. The first two inequalities is equivalent to
$ t- \frac{s+2\ell-2}{12}    \le k  \le t-\frac{s+\ell-2}{6}$, which implies $s< 2$ and hence $s=0$.
Now it is easy to see that $\ell $ has to be a multiple of $6$ and $k=t-\frac{\ell-2}{6}$.
By listing all such partitions and computing their dual, we obtain
$$
E=\{(6t+2,6t+2,0,0), (6t+2,6t,2,0), (6k+2,6k-2,4,0), \cdots, (6k+2,4k+2,2k,0)\}.
$$
Then one easily checked that
$$E'=\{(6t+3,6t+2,0,0), (6t+3,6t,2,0), (6k+3,6k-2,4,0), \cdots, (6k+3,4k+2,2k,0)\}$$ and $E'\subseteq (6t+3)\bigoplus S_{3,6t+3}$.
Hence we have $S_{4,6t+3}=S_{4,6t+2} \cup ((6t+3)\bigoplus S_{3,6t+3}) \setminus E'$ as desired.

The other cases are similar. We omit the details and only give some data.

Case 4. To obtain $S_{4,6t+4}$ from $S_{4,6t+3}$, we have
%
$$E'=\{(6t+4,6t+3,0,0), (6t+4,6t+1,2,0), (6k+4,6k-1,4,0), \cdots, (6k+4,4k+3,2k,0)\}$$ and $E'\subseteq (6t+4)\bigoplus S_{3,6t+4}$.
Hence we have $S_{4,6t+4}=S_{4,6t+3} \cup ((6t+4)\bigoplus S_{3,6t+4}) \setminus E'$ as desired.

Case 5. To obtain $S_{4,6t+5}$ from $S_{4,6t+4}$, we have
%
%
$$E'=\{(6t+5,6t+4,0,0), (6t+5,6t+2,2,0), (6k+5,6k,4,0), \cdots, (6k+5,4k+4,2k,0)\}$$ and $E'\subseteq (6t+4)\bigoplus S_{3,6t+4}$.
Hence we have $S_{4,6t+5}=S_{4,6t+4} \cup ((6t+5)\bigoplus S_{3,6t+5}) \setminus E'$ as desired.

Case 6. To obtain $S_{4,6t+6}$ from $S_{4,6t+5}$, we have
%
%
$$E'=\{(6t+6,6t+5,0,0), (6t+6,6t+2,2,0), (6k+6,6k,4,0), \cdots, (6k+6,4k+4,2k,0)\}$$ and $E'\subseteq (6t+6)\bigoplus S_{3,6t+6}$.
Hence we have $S_{4,6t+6}=S_{4,6t+5} \cup ((6t+6)\bigoplus S_{3,6t+6}) \setminus E'$ as desired.
\end{proof}

\section{Concluding Remark}
In this paper, we give an explicit order matching $\varphi$ for $L(3,n)$ using several different approaches. The methods
extend for $L(4,n)$. But for $L(m,n)$ with $m\ge 5$, we need new idea to construct the order matching.

We suspect that the greedy algorithm will succeed if we use an appropriate total ordering on $L(m,n)_i$. If succeeds, the corresponding chain tableau
will be helpful in finding the patterns. Existing algebraic proofs might give hints on the ordering.

\textbf{Acknowledgments:}
This work was supported by the National Natural Science Foundation of China (Nos. 12071311, 12571355, 12401441) and the Natural Science Foundation of Hunan Province (No. 2025JJ60010).
The authors declare that they have no conflicts of interest regarding the publication of this paper.

\textbf{Data Availability Statement:} No external data were used in this study. All relevant information is contained within the manuscript.

\end{document}